\documentclass{article}

\usepackage{amsmath,amssymb,amsfonts,amsthm}
\usepackage{fullpage}
\usepackage{color}
\usepackage{hyperref}
\usepackage{graphicx}

\newcommand{\q}[1]{[ {#1} ]_q}
\newcommand{\Q}[1]{\{ {#1} \}_{q^2}}
\newcommand{\binomq}[2]{ \binom{ #1 }{ #2 }_q}

\newcommand{\Uq}{\mathcal{U}_q(\mathfrak{gl}_n)}
\newcommand{\DOI}[1]{\href{http://dx.doi.org/#1}{DOI:#1}}
\newcommand{\arXiv}[1]{\href{http://arxiv.org/abs/#1}{arXiv:#1}}
\newcommand{\Z}{\mathbb{Z}}

\newtheorem{theorem}{Theorem}[section]
\newtheorem{proposition}[theorem]{Proposition}
\newtheorem{lemma}[theorem]{Lemma}

\newtheorem{definition}[theorem]{Definition}

\title{A Multi--species ASEP$(q,j)$  and $q$--TAZRP 
with Stochastic Duality}
\author{Jeffrey Kuan}
\date{\today}

\begin{document}
\maketitle

\abstract{ 
This paper introduces a multi--species version of a process called $\textrm{ASEP}(q,j)$. In this process, up to $2j$ particles are allowed to occupy a lattice site, the particles drift to the right with asymmetry $q^{2j} \in (0.1)$, and there are $n-1$ species of particles in which heavier particles can force lighter particles to switch places. Assuming closed boundary conditions, we explicitly write the reversible measures and a self--duality function, generalizing previously known results for two--species ASEP and single--species $\textrm{ASEP}(q,j)$.

Additionally, it is shown that this multi--species $\textrm{ASEP}(q,j)$ is dual to its space--reversed version, in which particles drift to the left. As $j\rightarrow\infty$, this multi--species $\textrm{ASEP}(q,j)$ converges to a multi--species $q$--TAZRP and the self--duality function has a non--trivial limit, showing that this multi--species $q$--TAZRP satisfies a space--reversed self--duality.

The construction of the process and the proofs are accomplished utilizing spin $j$ representations of $\mathcal{U}_q(\mathfrak{gl}_n)$, extending the approach used for single--species $\textrm{ASEP}(q,j)$.
}

\section{Introduction}

This paper will introduce a multi--species ASEP$(q,j)$ process which has stochastic self--duality. In order to motivate the problem, first review previous results in literature. 

The Asymmetric Simple Exclusion Process (ASEP) was introduced in \cite{MGP} and \cite{Sp}. This is an interacting particle system on the one--dimensional lattice $\mathbb{Z}$, where at most one particle is allowed to occupy a lattice site. Each particle has two exponential clocks, one for left jumps and one for right jumps. The clock for left jumps has rate $q$ and the clock for right jumps has rate $q^{-1}$, where $q \in (0,1)$ is the asymmetry parameter, and all clocks are independent of each other. When the clock rings, the particle makes the corresponding jump to the adjacent site unless that site is occupied, in which case the jump is blocked. 

In \cite{S}, Sch\"{u}tz shows that ASEP on the finite lattice $\{1,\ldots,L\}$ with closed boundary conditions (that is, boundary conditions in which particles can neither enter nor exit the lattice) satisfies a self--duality property. If a particle configuration is written as $\{\eta^x: 1 \leq x \leq L\}$, where $\eta^x \in \{0,1\}$ is the number of particles at lattice site $x$, then the self--duality is with respect to the function 
$$
D(\eta,\xi) = \prod_{x=1}^L  q^{2\xi^x\sum_{z=x+1}^L \eta^z   + 2\xi^x x} 1_{\eta^x \geq \xi^x}
$$
The proof uses the $\mathcal{U}_q(\mathfrak{gl}_2)$ symmetry of ASEP to show duality.

Because self--duality can be used to prove significant results about a system (e.g. for deriving long--time fluctuations \cite{BCS} or weak asymmetry convergence \cite{CST}), it is natural to attempt to construct other interacting particle systems with algebraic symmetry which can be used to prove duality. In \cite{CGRS}, the authors lay out a framework which takes a finite--dimensional Lie algebra $\mathfrak{g}$ and a representation $V$ and produces an interacting particle system and an explicit formula for the self--duality function and reversible measures. This is then done for the spin $2j$ representation of $\mathfrak{gl}_2$, and the resulting interacting particle system is named ASEP$(q,j)$. This is a generalization of ASEP in which up to $2j$ particles can occupy a lattice site, with asymmetry parameter $q^{2j}$. If $\eta^x \in \{0,1,\ldots,2j\}$ denotes the number of particles at lattice site $x$, then the self--duality function is
$$
D(\eta,\xi) = \prod_{x=1}^L \frac{\binomq{\eta^x}{\xi^x}}{\binomq{2j}{\xi^x}}q^{\xi^x\left( 2\sum_{z=x+1}^L \eta^z + \eta^x\right) + 4j\xi^x x} 1_{\eta^x \geq \xi^x}
$$
where $\binomq{\cdot}{\cdot}$ is a $q$--deformation of the usual binomial, defined in \eqref{QBinomial} below. When $j=1/2$, ASEP$(q,j)$ is the usual ASEP and the self--duality function reduces to Sch\"{u}tz's self--duality function, up to a constant.

This approach was also used later for $\mathcal{U}_q(\mathfrak{su}(1,1))$ \cite{CGRS2} and for $\mathcal{U}_q(\mathfrak{sp}_4)$ \cite{K}. The most relevant case here is for the spin $1/2$ representation of $\mathfrak{gl}_3$ in \cite{K}. The resulting interacting particle system is ASEP with second--class particles introduced in \cite{L}. Here, there are two classes (or species) of particles, with at most one particle per site, and first--class particles can force second--class particles to switch places. ASEP with second--class particles has two natural Markov projections: the projection onto the evolution of the first--class particles is the usual ASEP, and the projection onto the evolution of particle occupation is also usual ASEP. See Figure \ref{Images} for an illustrative example. If $\eta^x_i \in \{0,1\}$ denotes the number of $i$--th class particles at lattice site $x$, then then duality function is
$$
D(\eta,\xi) = \prod_{x=1}^L q^{2\xi_1^x \sum_{z=x+1}^L \eta^z_1 + 2 \xi_1^x x} 1_{\eta_1^z \geq \xi_1^x} q^{2\xi_2^x \sum_{z=x+1}^L (\eta^z_1+\eta^z_2) + 2\xi_2^x x} 1_{\eta_1^x + \eta_2^x \geq \xi_1^x + \xi_2^x}
$$
Similar results were also found in \cite{BS,BS2} using the Perk--Schultz quantum spin chain \cite{PS}.

\begin{figure}\label{Images}
\begin{center}
\caption{ASEP with second--class particles, where $\alpha=q$ and $\beta=q^{-1}$ are the left and right jump rates. The black particles are first--class particles and the red particles are second--class. If a first--class particle attempts to jump to a site occupied by a second--class particle, then the particles switch places. If a second--class particle attempts to jump to a site occupied by a first--class particle, the jump is blocked. 
The second image shows projection to the first--class particles and the third image shows projection to particle occupation. Both are usual ASEP. }
\includegraphics[height=2cm]{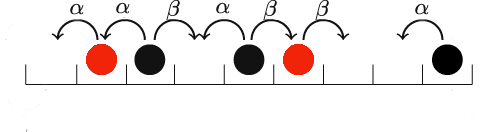}

\includegraphics[height=2cm]{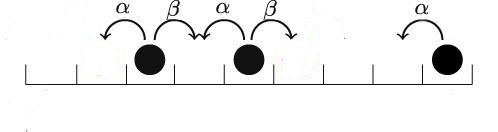}

\includegraphics[height=2cm]{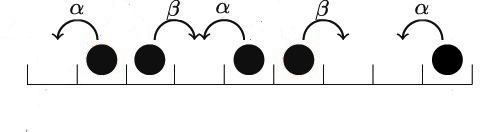}
\end{center}
\end{figure}

Here is an intuitive description of the result. If $\xi$ consists only of first--class particles, then the duality reduces to the Sch\"{u}tz duality for the projection onto the first--class particles. If $\xi$ consists only of second class particles, then the duality reduces to the Sch\"{u}tz duality for the projection onto particle occupation. Another way of describing the duality is that each first--class particle in $\xi$ counts first--class particles in $\eta$, and each second--class particle in $\xi$ counts both first and second--class particles in $\eta$. The interesting situation occurs when $\xi$ contains both first--class and second--class particles. 

With these previous results, it is natural to speculate that using the Lie algebra $\mathfrak{gl}_n$ and the spin $2j$ representation, the resulting interacting particle system should be a $(n-1)$--species ASEP$(q,j)$, in the sense that for each $1 \leq i \leq n-1$, the projection onto the first $i$ classes of particles is a $i$--species ASEP$(q,j)$. Furthermore, if the $\xi$ process in the self--duality function consists only of $i$th class particles, then the duality should reduce to the ASEP$(q,j)$ duality of \cite{CGRS} corresponding to the $i$th projection. Or, in other words, each $i$th class particle in $\xi$ counts $1,\ldots,i$th class particles in $\eta$ according to the ASEP$(q,j)$ duality. Indeed, this paper will prove that this is indeed true.

Additionally, in the $j\rightarrow\infty$ limit, the asymmetry parameter converges to $0$ and an arbitrary number of particles are allowed at each site. Indeed, the ASEP$(q,j)$ process converges to the $q$--TAZRP (Totally Asymmetric Zero Range Process)  introduced in \cite{SW}. A similar statement holds here: namely, the multi--species ASEP$(q,j)$ process in this paper converges to a multi--species $q$--TAZRP constructed in \cite{T}. Additionally, by applying a charge--parity symmetry, the self--duality function converges to a duality function between this multi--species $q$--TAZRP and its space--reversed version, in which particles jump in the opposite direction. 

During the preparation of this paper, the author learned in a private communication that an upcoming paper by V. Belitsky and G. Sch{\" u}tz \cite{BS3} analyzes the $n$--species ASEP with closed boundary conditions (corresponding to $j=1/2$ in the notation here) with $\mathcal{U}_{q}(\mathfrak{gl}_{n+1})$ symmetry and explicitly derives all invariant measures and a class of self--duality functions. Additionally, an application to the dynamics of shocks is proved. 

The reminder of this paper is organized as follows. Section \ref{Desc} defines the process and states the results for duality and reversible measures. Section \ref{RepDual} goes over the \cite{CGRS} framework and the necessary representation theory background. Section \ref{Constr} constructs the process using this framework. Section \ref{Results} proves the theorems concerning duality and reversible measures. 

\textbf{Acknowledgments.} Financial support was provided by the Minerva Foundation and NSF grant DMS-1502665. Part of the work was done during the workshop ``New approaches to non--equilibrium and random systems: KPZ integrability, universality, applications and experiments'' at the Kavli Institute for Theoretical Physics and was supported in part by NSF PHY11--25915. The author would like to thank Gioia Carinci, Ivan Corwin, Cristian Giardin\`{a}, Frank Redig, Tomohiro Sasamoto, Gunter Sch{\" u}tz, and Lauren Williams for valuable discussion.

\section{Description of Model and Statement of Results}\label{Desc}
Let $I=[a,b]\cap \Z\subset\Z$ be a (possibly infinite) connected interval in $\Z$, where $-\infty \leq a \leq b\leq \infty$. The state space $\mathcal{S}(n,j,I)$ of particle configurations is given by variables $\eta=(\eta_i^x: 1\leq i\leq n, x \in I)$ where each $\eta_i^x$ is a non--negative integer, and are subject to the condition that $\eta_1^x + \ldots + \eta_n^x = 2j$ for all $x$. One should think of $\eta_i^x$ as the number of $i$th class particles at site $x$ and $\eta_n^x$ as the number of holes. It will be convenient and useful to refer to holes as $n$th--class particles. For example, in the particle configuration below where $j=2,I=[1,4]\cap \Z$, some examples of $\eta_i^x$ are $\eta_3^1 = 1,\eta_1^2=2, \eta_4^4=2$.
$$
\begin{array}{cccc}
4 & 4 & 3 & 4\\
3 & 2 & 3 & 4 \\
2 & 1 & 2 & 3 \\
\underline{\text{ 1 }} & \underline{\text{ 1 }} &  \underline{\text{ 2 }} & \underline{\text{ 1 }}
\end{array}
$$
The classes of particles can also be viewed as having different masses, with the first--class particles being the heaviest and the $n$th--class particles being the lightest. The example above reflects the intuition that the heaviest particles should be closest to the ground.

For $1\leq k \leq l \leq n$, let 
$$
\eta_{[k,l]}^x = \eta_k^x + \ldots + \eta_l^x
$$
and note that $\eta_{[1,i]}^x + \eta_{[i+1,n]}^x=2j$ for all $i$ and $x$. Let $\emptyset$ denote the particle configuration with no particles, that is, $\emptyset^x_n=2j$ for all $x$. 

For a particle configuration $\xi$, let $\xi^{x \leftrightarrow x+1}_{i \leftrightarrow j}$ denote the particle configuration obtained by switching an $i$--th class particle at lattice site $x$ with a $j$--th class particle at lattice site $x+1$ (assuming that such a particle configuration exists).

Let $\q{n}$ and $\Q{n}$ denote the $q$--deformed integers
\begin{equation}\label{QDeformed}
\q{n} = \frac{q^{n}-q^{-n}}{q-q^{-1}} \quad \Q{n} = \frac{1-q^{2n}}{1-q^2}.
\end{equation}
Note that as $q\rightarrow 1$, both $\q{n}$ and $\Q{n}$ converge to $n$. For the remainder of this paper, $q$ will always be in the interval $(0,1)$.

\begin{definition} Suppose $f(\xi)$ is a function on $n$--species particle configurations $\xi \in \mathcal{S}(n,j,I)$ on the interval $I=[a,b]\cap \Z$, $-\infty \leq a \leq b \leq \infty$. The generator $\mathcal{L}$ of type $A_n$, spin $j$ ASEP on $I$ with closed boundary conditions acts on $f$ by
$$
\mathcal{L}f (\xi) = \sum_{x=a}^{b-1} \mathcal{L}_{x,x+1}f (\xi)
$$
where
\begin{multline*}
\mathcal{L}_{x,x+1}f (\xi) = \sum_{1 \leq k < l \leq n} \Big( q^{-1} \cdot q^{2(\xi_1^x +\ldots+\xi^x_{k-1})}\Q{\xi_k^x} \cdot q^{2(\xi^{x+1}_{l+1}+\ldots+\xi^{x+1}_n)} \Q{\xi^{x+1}_l}  \left(  f\left( \xi^{x \leftrightarrow x+1}_{k \leftrightarrow l} \right) - f(\xi)\right) \\
+ q \cdot q^{2(\xi^x_1+\ldots+\xi^x_{l-1})} \Q{\xi^x_l} \cdot q^{2(\xi^{x+1}_{k+1} + \ldots + \xi^{x+1}_n)} \Q{\xi_{k}^{x+1}} \left(  f\left( \xi^{x \leftrightarrow x+1}_{l \leftrightarrow k} \right) - f(\xi)\right) \Big)
\end{multline*}
\end{definition}

\textbf{Remark}. The process is named to emphasize the algebraic construction, since the Dynkin diagram of $\mathfrak{gl}_n$ is of type $A_n$. In principle, there could be other definitions of a multi--species ASEP$(q,j)$ process: for example, one could define a process where second--class particles can only jump from a site if there are no first--class particles at that site (see e.g. \cite{BK,BKS}). However, it is not obvious that all of them would necessarily satisfy a non--trivial duality, even if they satisfy an algebraic symmetry (e.g. \cite{M,NS}). 

The dynamics can be described through an ``inductive procedure''. Every lattice site has two clocks, one for left jumps and one for right jumps, with all clocks independent of each other. The clocks for left jumps are exponential with rate $q(1-q^2)^{-2}$ and the clocks for right jumps are exponential with rate $q^{-1}(1-q^2)^{-2}$. Because $q \in (0,1)$, this means the particles have a desire to drift to the right. There is a class system on the particles in which lighter particles cannot force heavier particles to move. Additionally, since each lattice site can accommodate precisely $2j$ particles, as a result the lower particles (that is, the particles closest to the ground) are given priority for right jumps in the following way.

When the right clock rings at lattice site $x$, the system asks the lowest particle if it would like to jump. With probability $1-q^2$ it says ``yes'' and with probability $q^2$ it says ``no.'' If it says ``no'', the system asks the next lowest particle if it would like to jump. Suppose that eventually, an $i$--th class particle decides to jump. However, it now needs to force a particle to switch places and move to the left. The particle starts with the highest particle (that is, the particle farthest from the ground). With probability $1-q^2$ the particle makes an attempt, and with probability $q^2$ the particle avoids an attempt and proceeds to the next--highest particle. If the particle makes an attempt at a $j$--th class particle, where $j>i$, then that $j$--th class particle moves to the left. If the $i$--th class particle makes an attempt at a $k$th class particle where $k \leq i$, then it is forced to return to the lattice site $x$, so no jump occurs. Observe that in this description, higher particles have a more difficult time moving to the right than lower particles do, because lower particles are asked to jump before higher particles.

\begin{figure}
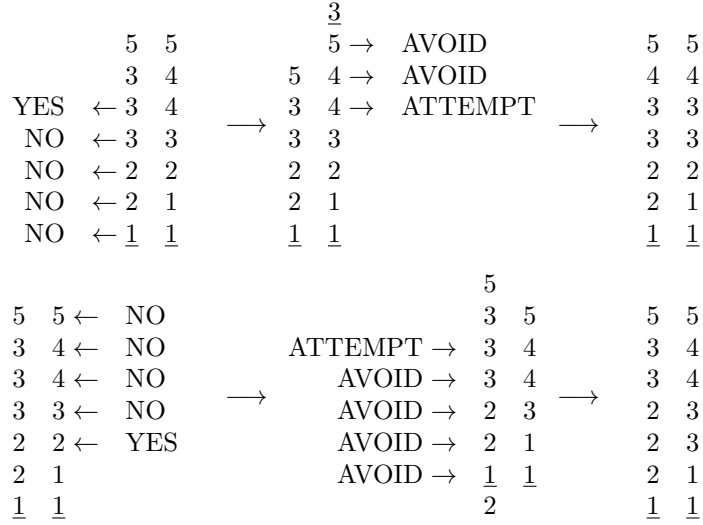

\caption{An example of the ``inductive procedure'' defined in the text. The top line shows a right jump where a third--class particle moves to the right and switches places with a fourth--class particle. The second line shows a left jump where a second--class particle jumps to the left and switches places with a third--class particle. Note that if it had attempted to switch with a first--class particle, then no jump occurs.}
$$
\begin{array}{rrc}
 & & \\
 & 5 & 5\\
 &  3 & 4\\
\text{YES} & \leftarrow 3 & 4\\
\text{NO} & \leftarrow 3 & 3\\
\text{NO} & \leftarrow 2 & 2\\
\text{NO} & \leftarrow 2 & 1\\
\text{NO}  & \leftarrow \underline{\text{1}} & \underline{\text{1}}\\
\end{array}
\quad
\longrightarrow
\begin{array}{cll}
 & \underline{3} & \\
 & 5\rightarrow & \text{AVOID} \\
 5 & 4\rightarrow & \text{AVOID} \\
 3 &  4\rightarrow & \text{ATTEMPT}\\
 3  & 3&\\
 2& 2&\\
 2& 1&\\
 \underline{1}& \underline{1}&
\end{array}
\longrightarrow
\begin{array}{rrc}
 & & \\
 & 5 & 5\\
& 4 & 4\\
 &  3 & 3\\
 &  3 & 3\\
 &  2 & 2\\
&  2 & 1\\
  &  \underline{\text{1}} & \underline{\text{1}}\\
\end{array}
$$
$$
\begin{array}{cll}
  & & \\
  5 & 5 \leftarrow & \text{NO} \\
  3 & 4 \leftarrow &\text{NO}\\
  3 & 4 \leftarrow &\text{NO}\\
  3 & 3 \leftarrow & \text{NO}\\
  2 & 2 \leftarrow &  \text{YES} \\
  2 & 1 &\\
  \underline{\text{1}} & \underline{\text{1}}\\
\end{array}
\quad
\longrightarrow
\begin{array}{rrc}
 &5 &   \\
 &3 & 5 \\
 \text{ATTEMPT} \rightarrow&3 & 4  \\
 \text{AVOID} \rightarrow&3 & 4\\
 \text{AVOID} \rightarrow&2 & 3\\
 \text{AVOID} \rightarrow&2 & 1\\
 \text{AVOID} \rightarrow&\underline{1}& \underline{1}\\
  & 2  & 
\end{array}
\longrightarrow
\begin{array}{rrc}
 & & \\
 & 5 & 5\\
& 3 & 4\\
 &  3 & 4\\
 &  2 & 3\\
 &  2 & 3\\
&  2 & 1\\
  &  \underline{\text{1}} & \underline{\text{1}}\\
\end{array}
$$
\end{figure}

Now suppose that the left jump clock rings at lattice site $x+1$. This occurrence is actually contrary to the particle's desire to move to the right. Therefore, in order to maintain the priority given to lower particles, the class system is momentarily \textit{reversed} for left jumps. That is, the system first asks the \textit{highest} particle if it would like to jump, with a jump occurring with probability $1-q^2$. If the particle declines to jump, the next highest particle is asked. Suppose that eventually an $i$th class particle attempts to jump to the left. It then attempts to move the \textit{lowest} particle to the right, with an avoidance of probability $q^2$. Note, however, that lighter particles cannot push heavier particles, so if it makes an attempt too soon then no jump actually occurs. 

It is not difficult to see that this verbal description matches the mathematical formula. For $k<l$, the rate of a $k$--th class particle at site $x$ jumping to the right and pushing a $l$--th class particle can be computed by
\begin{align*}
& [\text{rate of right jump clock}] \cdot [\text{all particles of class }1,\ldots,k-1 \text{ say no}] \cdot [\text{some class } k \text{ particle says yes}]  \\
& \quad \quad \cdot [\text{avoids all particles of class } n,n-1,\ldots,l+1] \cdot [\text{makes attempt at some class } l \text{ particle}] \\
& = q^{-1}(1-q^2)^{-2} \cdot q^{2(\xi^x_1 + \ldots + \xi^x_{k-1})} \cdot (1-q^2 + q^2(1-q^2) + \ldots + q^{2(\xi^x_{k}-1)}(1-q^2)) \\
& \quad \quad \cdot q^{2(\xi_n^{x+1}+\ldots+\xi_{l+1}^{x+1})} \cdot (1-q^2 + q^2(1-q^2) + \ldots + q^{2(\xi^{x+1}_{l}-1)}(1-q^2))\\
& = q^{-1} \cdot q^{2(\xi^x_1 + \ldots + \xi^x_{k-1})} \Q{\xi^x_k} \cdot q^{2(\xi_n^{x+1}+\ldots+\xi_{l+1}^{x+1})}  \Q{\xi^{x+1}_l}.
\end{align*}
An identical argument holds for left jumps.

From this description of the process, note a few basic properties, which will be proved rigorously later in the paper.

\begin{itemize}
\item
There are two natural Markov projections. Note that lighter particles cannot influence the behavior of heavier particles. Therefore, the projection onto the first $i$ classes by treating the $i+1,\ldots,n-1$ classes as holes is Markov. A second projection that demonstrates this lack of influence is to treat the $i,i+1,\ldots,n-1$ classes all as identical $i$th--class particles.

To mathematically define the projections, set for any map $\sigma:\{1,\ldots,n\}\rightarrow \{1,\ldots,m\}$ the projection $\Pi^{\sigma}:\mathcal{S}(n,j,I)\rightarrow \mathcal{S}(m,j,I)$ 
$$
\left( \Pi^{\sigma}\eta \right)^x_k = \sum_{l \in \sigma^{-1}(k)} \eta^x_l, \quad 1\leq k \leq m.
$$
where the sum is zero if $\sigma^{-1}(k)$ is empty.

Then the first projection is $\Pi^n_{i+1}:=\Pi^{\sigma}$ where $\sigma:\{1,\ldots,n\}\rightarrow \{1,\ldots,i+1\}$ is defined by 
$$
\sigma(l) 
= 
\begin{cases}
l, \quad & 1 \leq l \leq i, \\
i+1, \quad & i+1 \leq l \leq n,
\end{cases}
$$
and the second projection is $\tilde{\Pi}^n_{i+1}:=\Pi^{\sigma}$ where $\sigma:\{1,\ldots,n\}\rightarrow \{1,\ldots,i+1\}$ is defined by 
$$
\sigma(l) 
= 
\begin{cases}
l, \quad & 1 \leq l \leq i-1, \\
i, \quad & i \leq l \leq n-1,\\
i+1, \quad & l=n.
\end{cases}
$$



In the symmetric case $q=1$, the projection $\Pi^{\sigma}$ is Markov for every $\sigma$. These statements will be proved in Proposition \ref{MarkovProjection}.

\item
Because the left jump rates are determined from the right jump rates by reversing the order of the particles, the system satisfies the so--called \textit{charge--parity} (CP) symmetry. This means that making the substitutions in the lattice $x\mapsto L+1-x$ and in the class system $i\mapsto n+1-i$ preserves the generator. The CP--symmetry will later be used in finding a duality in the $j\rightarrow\infty$ limit. These statements will be proven in Proposition \ref{CPS}. 

\end{itemize}

It is natural to consider the symmetric limit, when $q\rightarrow 1$. The result is a multi--species version of the $2j$--SEP from \cite{GKRV}. For other single--species examples of symmetric processes with duality, see \cite{GRV} and \cite{SS}.

\begin{definition}\label{SSEP} The generator of type $A_n$, spin $j$ SSEP on $I=[a,b]\cap \mathbb{Z}$ is
$$
\mathcal{L}f(\xi) = \sum_{x=a}^{b-1} \sum_{1 \leq k < l \leq n}\left( \xi_k^x \xi_l^{x+1}\left( f(\xi_{k \leftrightarrow l}^{x \leftrightarrow x+1} ) - f(\xi) \right) + \xi_l^x \xi_k^{x+1} \left( f(\xi_{l \leftrightarrow k}^{x \leftrightarrow x+1} ) - f(\xi) \right) \right)
$$
\end{definition}
Observe that the generator is preserved under any permutation of the classes of the particles. Indeed, because the system satisfies the charge--parity symmetry, and also satisfies parity symmetry when $q=1$, it must satisfy charge symmetry as well. 

Now take the $j\rightarrow\infty$ limit in the definition of the generator of type $A_n$, spin $j$ ASEP. Restricting to particle configurations with only finitely many particles at each lattice site, this limit corresponds to taking all $\xi_n^x \rightarrow \infty$. From either the mathematical definition or the verbal description, one sees that the particles can only jump to the right. In fact, the result is a (time--rescaled) multi--species $q$--TAZRP (Totally Asymmetric Zero Range Process). This process was previously introduced in \cite{T}, which used a deformation of the affine Hecke algebra of type $GL_L$. 

\begin{definition}\label{Tak} \cite{T} The generator of type $A_n$ $q$--TAZRP on $I=[a,b]\cap \mathbb{Z}$ is
$$
\mathcal{L}f(\xi) =  \sum_{x=a}^{b-1} \sum_{i=1}^{n-1} q^{2(\xi^x_1+\ldots+\xi^x_{i-1})}\Q{\xi^x_i}  \left( f\left( \xi^{x \leftrightarrow x+1}_{i \leftrightarrow n}\right) - f(\xi) \right)
$$
\end{definition}

Taking both limits $q\rightarrow 1$ and $j\rightarrow\infty$ will depend on the order of the limits. In the $j\rightarrow\infty$ limit of Definition \ref{SSEP}, rescaling time by $2j$ results in $n-1$ independent symmetric zero range processes. Taking $q\rightarrow 1$ in Definition \ref{Tak} results in $n-1$ independent $1$--TAZRPs.

With the processes defined, we now proceed to the duality properties. Recall the definition of duality, which can be defined for Markov processes in general.

\begin{definition} Two Markov processes $X(t)$ and $Y(t)$ on state spaces $\mathfrak{X}$ and $\mathfrak{Y}$ are dual with respect to a function $D$ on $ \mathfrak{X}\times \mathfrak{Y}$ if
$$
\mathbb{E}_x\left[  D(X(t),y) \right] = \mathbb{E}_y \left[ D(x,Y(t))\right] \text{ for all } (x,y) \in \mathfrak{X}\times \mathfrak{Y} \text{ and all } t \geq 0.
$$
On the left--hand--side, the process $X(t)$ starts at $X(0)=x$, and on the right--hand--side the process $Y(t)$ starts at $Y(0)=y$.

An equivalent definition is that if the generator $\mathcal{L}_X$ of $X(t)$ is viewed as a $\mathfrak{X} \times \mathfrak{X}$ matrix, the generator $\mathcal{L}_Y$ of $Y(t)$ is viewed as a $\mathfrak{Y}\times \mathfrak{Y}$ matrix, and $D$ is viewed as a $\mathfrak{X} \times \mathfrak{Y}$ matrix, then $\mathcal{L}_XD = D\mathcal{L}_Y^*$.

If $X(t)$ and $Y(t)$ are the same process, in the sense that $\mathfrak{X}=\mathfrak{Y}$ and $\mathcal{L}_{{X}} = \mathcal{L}_{{Y}}$, then we say that $X(t)$ is self--dual with respect to $D$.
\end{definition}

Note that in the literature, there are slightly different conceptions of what it means for two Markov processes to be the ``same.'' For example, some authors allow for a bijection $\mathfrak{X} \rightarrow \mathfrak{Y}$ which sends $\mathcal{L}_X$ to $\mathcal{L}_Y$, and other authors require the state spaces to be irreducible.

\begin{theorem}\label{DualityTheorem}
(a)
The type $A_n$, spin $j$ process on $I \subseteq \mathbb{Z} $ is self--dual with respect to 
$$
D(\eta,\xi)= \prod_{x\in I} \q{\eta_1^x}^!  \prod_{i=1}^{n-1}   1_{\eta_{[1,i]}^x \geq \xi_{[1,i]}^x}   \frac{  \q{  \eta_{[1,i+1]}^x - \xi_{[1,i]}^x   }^!  }{\q{\eta_{[1,i]}^x   - \xi_{[1,i]}^x }^! }  q^{4jx\xi_i^x} q^{ \xi_i^x (\sum_{z>x} 2\eta_{[1,i]}^z + \eta_{[1,i]}^x)}
$$

(b) Space--reversed type $A_n$, spin $j$ ASEP is dual to type $A_n$, spin $j$ ASEP with respect to the function
$$
D(\eta,\xi)=\prod_{x\in I} \q{2j-\eta_{[1,n-1]}^x}^!  \prod_{i=1}^{n-1}   1_{2j -\eta_{[1,n-i]}^x \geq \xi_{[1,i]}^x}   \frac{  \q{  2j - \eta_{[1,n-i-1]}^x - \xi_{[1,i]}^x   }^!  }{\q{2j - \eta_{[1,n-i]}^x   - \xi_{[1,i]}^x }^! }   q^{ -\xi_i^x (\sum_{z>x} 2\eta_{[1,n-i]}^z + \eta_{[1,n-i]}^x)} 
$$

Space--reversed type $A_n$ $q$--TAZRP is dual to type $A_n$ $q$--TAZRP with respect to the function
$$
D(\eta,\xi)=\prod_{x\in I} \prod_{i=1}^{n-1}      q^{ \xi_i^x (\sum_{y \leq x} 2\eta_{[1,n-i]}^y) } .
$$
Here, the $\eta$ process has particles jumping to the left, while the $\xi$ process has particles jumping to the right.

\end{theorem}

\textbf{Remarks}. If $\xi$ has infinitely many particles, then $D(\eta,\xi)$ is identically zero and the theorem is still true, albeit trivial.

The duality in (a) generalizes previous dualities. When $n=2,j=1/2$ it reduces to \cite{S}, when $n=2$ and $j$ is arbitrary it reduces to \cite{CGRS} (for $q \in (0,1)$) and \cite{GKRV} (for $q=1$), and when $n=3,j=1/2$ it reduces to \cite{BS},\cite{K}.

When $n=2,j=1/2$, the duality in (b) is the same as in \cite{IS}. When $n=2$, the duality for this multi--species $q$--TAZRP reduces to that of \cite{BCS}. It is mentioned in Remark 3.6 of \cite{CGRS} that the duality functions converge to $0$ as $j\rightarrow\infty$, and no obvious re--normalization seems to result in a non--trivial self--duality function. And indeed, due to the \textit{total} asymmetry in the $q$--TAZRP, it is necessary to consider a space--reversed process. To see this, suppose that in the initial conditions for $\eta$ and $\xi$, the particles are contained in $[2,L-1]$. If both processes had particles jump to the right, then at large times all particles will occupy lattice site $L$. If the $\eta$ process evolves with $\xi$ fixed, then the duality function counts none of the particles. However, if the $\xi$ process evolves with $\eta$ fixed, then the duality function counts all of the particles. This violates the definition of duality, but this contradiction is avoided (in this case) if one of the processes is space--reversed.

The method of \cite{CGRS} also results in formulas for the reversible measures if $I$ is finite:

\begin{proposition} For any $\vec{m}=(m_1,\ldots,m_{n-1})$, let $\mathcal{S}_{\vec{m}}(n,j,I) \subset \mathcal{S}(n,j,I)$ denote the particle configurations on a finite interval $I \subseteq \mathbb{Z}$ with precisely $m_i$ particles of class $i$. Then the probability measure $\mathbb{P}_{\vec{m}}(\xi)$ on $\mathcal{S}(n,j,I)$ defined by 
$$
\mathbb{P}_{\vec{m}}(\xi) = 1_{ \xi \in \mathcal{S}_{\vec{m}}(n,j,I)}(Z_{\vec{m}}(n,j,I))^{-1}  \prod_{x\in I}  \prod_{i=1}^{n}  \frac{1}{\q{\xi_i^x}^!} q^{ ( \xi_i^x)^2/2 } \prod_{\substack{(y,x) \in I^2 \\ y < x}} \prod_{i=1}^{n-1}  q^{ -   2 \xi_{i+1}^y \xi_{[1,i]}^x  },
$$
where $Z_{\vec{m}}(n,j,I)$ is a normalizing constant, is a reversible measure of type $A_n$, spin $j$ ASEP. Any reversible measure is a convex combination of $\mathbb{P}_{\vec{m}}(\xi)$. 
\end{proposition}

\textbf{Remarks.} In the $j\rightarrow\infty$ limit, the normalizing constant $Z_{\vec{m}}(n,j,I)$ converges to $0$. Indeed, because type $A_n$ $q$--TAZRP is totally asymmetric, it does not possess a reversible measure. When $q=1$, one obtains reversible product measures for multi--species $2j$--SEP. Also note that if $I$ is infinite, then in general $Z_{\vec{m}}(n,j,I)$ is zero, even if all $m_i$ are finite.

In the cases when $n=2,j=1/2$,  or $n=2$, $j$ arbitrary, or $n=3,j=1/2$, it is not hard to see that this expression reduces to previously known formulas for reversible measures \cite{BS,CGRS}. There are previous results (e.g. \cite{C2,CMW,M,MV,Uch}) which classify the stationary measures for multi--species ASEP with open boundary conditions, in which some particles are allowed to enter and exit the lattice. The $\mathbb{P}_{\vec{m}}(\xi)$ here must correspond to some degeneration of these results, but such an identification is not made here.

\section{From Representation Theory to Stochastic Duality}\label{RepDual}
This section will introduce the necessary background for the remainder of the paper. In \cite{CGRS}, a general framework is described for constructing an interacting particle system with stochastic duality from the representation theory of quantum groups. The approach is as follows:

\begin{itemize}
\item
Start with a quantization $\mathcal{U}_q(\mathfrak{g})$ of a Lie algebra $\mathfrak{g}$, and a central element $C\in \mathcal{U}_q(\mathfrak{g})$.

\item For a fixed representation $V$ of $\mathcal{U}_q(\mathfrak{g})$, take its weight space decomposition 
$$
V = \bigoplus_{\mu} V[\mu].
$$
Intuitively, each $V[\mu]$ is a generalized eigenspace, and each $\mu$ is a generalized eigenvalue. Pick a basis $\{v_{\mu}\}$ of $V$ where each $v_{\mu} \in V[\mu]$, and identify each $v_{\mu}$ with a particle configuration at a single lattice site.

\item Fix a choice of the co--product $\Delta: \mathcal{U}_q(\mathfrak{g}) \rightarrow \mathcal{U}_q(\mathfrak{g}) \otimes \mathcal{U}_q(\mathfrak{g})$.   Let $h$ be the linear operator on $V\otimes V$ defined by $\Delta(C)$.  Define a quantum Hamiltonian $H$ on $V^{\otimes L}$ by 
$$
H = \sum_{x=1}^{L-1} h^{x,x+1}
$$
where $h^{x,x+1}$ acts $h$ on the $x$--th and $(x+1)$--th component of $V^{\otimes L}$. Extend the basis of $V$ to a basis of $V^{\otimes L}$, and identify each $v_{\mu^1} \otimes \ldots \otimes v_{\mu^L}$ with a particle configuration on the one--dimensional lattice $\{1,\ldots,L\}$. For ease of notation, we will take $I=\{1,\ldots,L\}$ for the remainder of the paper, and there will be no loss of generality in the proofs -- by a standard exercise the Markov process is well--defined on the infinite lattice because the jump rates are uniformly bounded (e.g. Chapter 1 of \cite{Ligg}).

\item Suppose there is a positive ground state $g\in V^{\otimes L}$, that is, a $g$ satisfying $Hg=0$ where the coefficient of each basis vector $v_{\mu^1} \otimes \ldots \otimes v_{\mu^L}$ in $g$ is positive. Let $G$ be the diagonal matrix where each nonzero entry is the coefficient in $g$. Defining $\mathcal{L}=G^{-1}HG$, it follows from the construction that each row of $\mathcal{L}$ sums to zero. If the off--diagonal entries are nonnegative, then $\mathcal{L}$ is the generator for a continuous time Markov process. 

\item Suppose there is a diagonal matrix $B$ such that $B^{-1}HB$ is self--adjoint. For any operator $S$ on $V^{\otimes L}$ such that $SH=HS$, it follows from the definition of duality that $G^{-1}SG^{-1}B^2$ is a self--duality function for $\mathcal{L}$. Additionally, $G^2B^{-2}$ is a reversible measure of $\mathcal{L}$.

\item If there is a generator $\tilde{\mathcal{L}}$ on the same state space as $\mathcal{L}$ such that $\mathcal{L}=V^{-1}\tilde{\mathcal{L}}V$ and $D$ is a self--duality of $\mathcal{L}$, then $VD$ is a duality between $\tilde{\mathcal{L}}$ and $\mathcal{L}$.

\end{itemize}

The remainder of this section will describe each of these six items before proceeding to explicit calculations.

\subsection{Definition of Quantum Groups}

This subsection defines the quantum groups. See \cite{J} for a more thorough treatment. 

Let $q$ be a formal variable. The quantum group $\Uq$ is the Hopf algebra with generators  $\{ E_{i,i+1},E_{i+1,i}: 1\leq i \leq n-1\},\{q^{E_{ii}}:1\leq i\leq n\}$ satisfying the relations (below, $q^{E_{ii}}$ are all invertible and the multiplication is written additively in the exponential, so for example $q^{-E_{11} + 2E_{22}}= \left( q^{E_{11}}\right)^{-1} \left( q^{E_{22}}\right)^2$) 
$$
q^{E_{ii}}q^{E_{jj}} = q^{E_{jj}} q^{E_{ii}} = q^{E_{ii}+E_{jj}}
$$
$$
[E_{i,i+1},E_{i+1,i}] = \frac{q^{E_{ii}-E_{i+1,i+1}}-q^{E_{i+1,i+1}-E_{ii}}}{q-q^{-1}} \quad  \quad  [E_{i,i+1}E_{j+1,j}] =0, \quad i\neq j
$$
\begin{align*}
q^{E_{ii}}E_{i,i+1} &= q E_{i,i+1}q^{E_{ii}} \quad \quad q^{E_{ii}}E_{i-1,i} = q^{-1} E_{i-1,i} q^{E_{ii}}  \quad \quad [q^{E_{ii}},E_{j,j+1}]=0, \quad j\neq i,i-1\\
q^{E_{ii}}E_{i,i-1} &= q E_{i,i-1}q^{E_{ii}} \quad \quad q^{E_{ii}}E_{i+1,i} = q^{-1} E_{i+1,i} q^{E_{ii}} \quad \quad [q^{E_{ii}},E_{j,j-1}]=0, \quad j\neq i,i+1\
\end{align*}
\begin{align*}
E_{i,i+1}^2E_{j,j+1} - (q+q^{-1})E_{i,i+1}E_{j,j+1}E_{i,i+1} +  E_{j,j+1}E_{i,i+1}^2 = 0 , \quad & i = j\pm 1\\
E_{i,i-1}^2E_{j,j-1} - (q+q^{-1})E_{i,i-1}E_{j,j-1}E_{i,i-1} +  E_{j,j-1}E_{i,i-1}^2 = 0 , \quad & i = j\pm 1\\
[E_{i,i+1},E_{j,j+1} ]= 0 = [E_{i,i-1},E_{j,j-1}] , \quad & i \neq j\pm 1
\end{align*}

For any $1 \leq i \neq j \leq n$, define $E_{ij}$ inductively by
$$
E_{ij} = E_{ik}E_{kj} - q^{-1}E_{kj}E_{ik}, \quad i < k < j \text{ or } i > k > j.
$$
This definition does not depend on the choice of $k$. From \cite{GZB}, there is a central element
\begin{equation}\label{GZBC}
 \sum_{i=1}^n q^{2i-2n-1}q^{2E_{ii}}+ (q-q^{-1})^2 \sum_{1 \leq i < j \leq n} q^{2j-2n-2}q^{E_{ii}+E_{jj}}E_{ij}E_{ji}.
\end{equation}

\subsection{Representations}

For $2j\in \mathbb{Z}$, consider the spin $2j$ representation of $\mathcal{U}_q(\mathfrak{gl}_n)$. This is the finite--dimensional representation with dimension $\binom{2j+n-1}{n-1}$. The dimension is the number of ways to place $n$ species of particles at a single lattice site (in other words, $\binom{2j+n-1}{n-1}= \vert \mathcal{S}(n,j,1) \vert$). Therefore, a basis can be indexed by $\{v_{\mu}\}$ where $\mu$ is of the form
$$
\{\mu=(\mu_1,\mu_2,\ldots,\mu_n): \mu_1+\ldots+\mu_n=2j, \  \mu_i \geq 0\}
$$
One can think of this as $\mu_i$ particles of type $i$, where type $n$ particles are holes. Let $\Omega$ denote the vacuum vector, that is, 
$$
\Omega = v_l \otimes \cdots \otimes v_l 
$$
where $l=(0,0,\ldots,0,2j)$. This is the state with no particles. 

For $1 \leq i< j \leq n$, define
\begin{align*}
\mu_{j \rightarrow i}  &= (\mu_1,\ldots,\mu_{i-1}, \mu_i + 1, \mu_{i+1}, \ldots, \mu_{j-1}, \mu_j - 1, \mu_{j+1}, \ldots, \mu_n),\\  
\mu_{i \rightarrow j} &= (\mu_1,\ldots,\mu_{i-1}, \mu_i - 1, \mu_{i+1}, \ldots, \mu_{j-1}, \mu_j + 1, \mu_{j+1}, \ldots, \mu_n).
\end{align*}
So $\mu_{j \rightarrow i}$ represents the particle configuration where a $j$ class particle has been replaced with an $i$ class particle. Recall from \eqref{QDeformed} that the $q$--integers $\q{n}$ are defined by
$
\q{n} = \tfrac{q^n-q^{-n}}{q-q^{-1}}.
$
Extend these definitions to 
\begin{equation}\label{QBinomial}
\q{n}^! = \q{1} \cdots \q{n}, \quad \quad \binomq{n}{m} = \frac{ \q{n}^! }{ \q{n-m}^! \q{m}^!},
\end{equation}
where by convention $\q{0}^!=1$.
Note that these satisfy the relations
\begin{align}\label{Weyl}
\q{n+1}\q{m} &= \frac{q^{n+1+m}-q^{n+1-m}-q^{m-n-1}+q^{-n-1-m}}{(q-q^{-1})^2}, \notag \\
\q{n}\q{m+1} &= \frac{q^{n+m+1}-q^{n-m-1}-q^{m+1-n}+q^{-n-1-m}}{(q-q^{-1})^2}, \notag\\
\q{n+1}\q{m}-\q{n}\q{m+1} &= \frac{q^{n-m-1}+q^{m+1-n} - q^{n+1-m}-q^{m-n-1}}{(q-q^{-1})^2} = \frac{q^{m-n}-q^{n-m}}{q-q^{-1}} = \q{m-n},
\end{align}
and also for $n>1$,
\begin{equation}\label{Serre}
\q{n+1}-(q+q^{-1})\q{n} + \q{n-1} =0,
\end{equation}
and
\begin{equation}\label{Lusztig}
\q{n+1}-q^{-1}\q{n}=q^n.
\end{equation}
It will also be useful to define $\{n\}_q$ by
$$
\{n\}_q = \frac{1-q^n}{1-q},
$$
which satisfy $q^{n-1}\q{n} = \Q{n}$. There is also the identity
\begin{equation}\label{Projection}
 \Q{k+l} =   q^{2l}\Q{k} + \Q{l} 
\end{equation}
which extends inductively to a telescoping sum:
\begin{equation}\label{Telescope}
\Q{m_1+\ldots+m_{i-1}} = \sum_{r=1}^{i-1} q^{2(m_{r+1} + \ldots + m_{i-1})}\Q{m_r} .
\end{equation}
Similarly,
\begin{equation}\label{Telescope2}
\sum_{s=j+1}^n q^{2s-1} q^{2(m_{j+1}+\ldots+m_{s-1})} \Q{m_s+1} = \frac{q^{-1} \left( q^{2(j+1)} - q^{2n+2}q^{2(m_{j+1}+\ldots+m_n)} \right)}{1-q^2}.
\end{equation}

An explicit action of the generators of $\Uq$ on this representation is given in the following lemma.

\begin{lemma} The action
\begin{equation*}
E_{i,i+1} v_{\mu} = \q{\mu_{i+1}} v_{\mu_{i+1\rightarrow i}} \quad E_{i+1,i} v_{\mu} = \q{\mu_i} v_{\mu_{i \rightarrow i+1}} \quad q^{E_{ii}} v_{\mu} = q^{\mu_i} v_{\mu}
\end{equation*}
is a representation of $\mathcal{U}_q(\mathfrak{gl}_n)$. Furthermore, for $i<j$ the actions of $E_{ij}$ and $E_{ji}$ are given by
\begin{equation}\label{LusztigAction}
\begin{aligned}
E_{ij}v_{\mu} &= q^{\mu_{i+1}+\mu_{i+2}+\ldots+\mu_{j-1}}\q{\mu_j}v_{\mu_{j\rightarrow i}} \\
E_{ji}v_{\mu} &= q^{\mu_{i+1}+\mu_{i+2}+\ldots+\mu_{j-1}}\q{\mu_i}v_{\mu_{i\rightarrow j}} 
\end{aligned}
\end{equation}
\end{lemma}
\begin{proof} The proof consists of checking that the relations defining $\Uq$ still hold in the representation.
By \eqref{Weyl},
$$
[E_{i,i+1},E_{i+1,i}]v_{\mu} =  \left(\q{\mu_{i+1}+1}\q{\mu_i} - \q{\mu_i+1}\q{\mu_{i+1}}\right)v_{\mu} = \frac{q^{E_{ii}-E_{i+1,i+1}} - q^{E_{i+1,i+1}-E_{ii}} }{q-q^{-1}}v_{\mu},
$$
so the relation on $[E_{i,i+1},E_{i+1,i}]$ holds. For $1\leq i\leq n-2,$
\begin{align*}
&\left(E_{i,i+1}^2 E_{i+1,i+2} - (q+q^{-1})E_{i,i+1} E_{i+1,i+2}E_{i,i+1} + E_{i+1,i+2}E_{i,i+1}^2\right)v_{\mu} \\
&= \left( \q{\mu_{i+1}}\q{\mu_{i+1}+1}\q{\mu_{i+2}} - (q+q^{-1}) \q{\mu_{i+1}}\q{\mu_{i+2}}\q{\mu_{i+1}} + \q{\mu_{i+2}}\q{\mu_{i+1}-1}\q{\mu_{i+1}} \right)v_{\mu}\\
&= \q{\mu_{i+2}}\q{\mu_{i+1}} \left( \q{\mu_{i+1}+1}   - (q+q^{-1})\q{\mu_{i+1}} + \q{\mu_{i+1}-1} \right) v_{\mu}
\end{align*}
so if $\mu_{i+1}=0$ then this is zero because $\q{\mu_{i+1}}=0$, and else it is zero by \eqref{Serre}. The remaining relations are similar to check.

To prove \eqref{LusztigAction}, use induction and \eqref{Lusztig}:
\begin{align*}
E_{ij}v_{\mu} &= E_{i,j-1} \q{\mu_j}v_{\mu_{j\rightarrow j-1}} - q^{-1}E_{j-1,j}E_{i,j-1}v_{\mu} \\
&= (q^{\mu_{i+1}+\mu_{i+2}\ldots+\mu_{j-2}}\q{\mu_{j-1}+1}\q{\mu_j} - q^{-1}\q{\mu_j}q^{\mu_{i+1}+\mu_{i+2}\ldots+\mu_{j-2}}\q{\mu_{j-1}})v_{\mu_{j \rightarrow i}}\\
&= q^{\mu_{i+1}+\mu_{i+2}+\ldots+\mu_{j-1}}\q{\mu_j}v_{\mu_{j \rightarrow i}}. 
\end{align*}
with the argument for $E_{ji}v_{\mu}$ being similar. 
\end{proof} 

The following equations, which are not hard to prove, will be used later in the paper.
\begin{align}\label{Binomial}
\frac{E_{i,i+1}^m}{\Q{m}^!} v_{\mu}&= \frac{\q{\mu_{i+1}}\cdots\q{\mu_{i+1}-m+1}}{\Q{m}^!}v_{\mu+m\epsilon_i} 
 =\frac{  \q{\mu_{i+1}}\cdots\q{\mu_{i+1}-m+1} }{q^{m(m-1)/2}\q{m}^!}v_{\mu+m\epsilon_i} \notag \\
&= q^{m/2}q^{-m^2/2} \binom{\mu_{i+1}}{m}_qv_{\mu+m\epsilon_i}.
\end{align}
and
\begin{equation}\label{LusztigAction2}
\begin{aligned}
q^{E_{jj}}E_{ij}v_{\mu} &= q^{\mu_{i+1}+\mu_{i+2}+\ldots+\mu_{j-1}}\Q{\mu_j} v_{\mu+\epsilon_i+\ldots+\epsilon_{j-1}}, \\
q^{E_{ii}}E_{ji}v_{\mu} &= q^{\mu_{i+1}+\mu_{i+2}+\ldots+\mu_{j-1}}\Q{\mu_i} v_{\mu-\epsilon_i-\ldots-\epsilon_{j-1}}, \\
q^{2E_{jj}}E_{ij}v_{\mu} &= q^{-1} \cdot q^{\mu_{i+1}+\mu_{i+2}+\ldots+\mu_{j}}\Q{\mu_j} v_{\mu+\epsilon_i+\ldots+\epsilon_{j-1}}, \\
q^{2E_{ii}}E_{ji}v_{\mu} &= q^{-1}\cdot q^{\mu_{i}+\mu_{i+2}+\ldots+\mu_{j-1}}\Q{\mu_i} v_{\mu-\epsilon_i-\ldots-\epsilon_{j-1}}, \\
q^{E_{ii}+E_{jj}}E_{ij}v_{\mu} &= q \cdot q^{\mu_{i}+\mu_{i+1}+\ldots+\mu_{j-1}}\Q{\mu_j} v_{\mu+\epsilon_i+\ldots+\epsilon_{j-1}}, \\
q^{E_{ii}+E_{jj}}E_{ji}v_{\mu} &= q \cdot q^{\mu_{i+1}+\ldots+\mu_{j}} \Q{\mu_i} v_{\mu-\epsilon_i-\ldots-\epsilon_{j-1}}.
\end{aligned}
\end{equation}

\subsection{Quantum Hamiltonian}

The co--product is an algebra morphism $\Delta: \Uq \rightarrow \Uq \otimes \Uq$ defined by
\begin{align*}
\Delta\left(q^{E_{ii}}\right) &= q^{E_{ii}} \otimes q^{E_{ii}}, \\
\Delta\left( E_{i,i+1} \right) &= q^{E_{ii}-E_{i+1,i+1}} \otimes E_{i,i+1} + E_{i,i+1} \otimes 1, \\
\Delta\left( E_{i,i-1} \right) &= 1 \otimes E_{i,i-1} + E_{i,i-1} \otimes q^{E_{i+1,i+1}-E_{ii}}.
\end{align*}
Note that unless $q = 1$, $\Delta$ does not satisfy co--commutativity. In other words, if $P$ is the permutation $P(a\otimes b)=b\otimes a$, then $P\circ\Delta \neq \Delta$. This is the algebraic explanation for why the asymmetry occurs. However, the co--product does satisfy the co--associativity property
$$
(\mathrm{id} \otimes \Delta) \circ \Delta = (\Delta \otimes \mathrm{id}) \circ \Delta,
$$
so that there is a well--defined algebra morphism $\Delta^{(m-1)} : \Uq \rightarrow \Uq^{\otimes m}$ satisfying $\Delta^{(m)} = (\mathrm{id} \otimes \Delta^{(m-1)}) \circ \Delta = (\Delta^{(m-1)} \otimes \mathrm{id}) \circ \Delta$. Explicitly, 
\begin{align*}
\Delta^{(m-1)}(E_{i,i+1}) &= \sum_{x=1}^m (q^{E_{ii}})^{\otimes x-1} \otimes E_{i,i+1} \otimes 1^{\otimes m-x}, \\
\Delta^{(m-1)}(E_{i+1,i}) &= \sum_{x=1}^m 1^{\otimes x-1} \otimes E_{i+1,i} \otimes \left(q^{-E_{ii}}\right)^{\otimes m-x}, \\
\Delta^{(m-1)}(q^{E_{ii}}) &= q^{E_{ii}} \otimes \cdots \otimes q^{E_{ii}}.
\end{align*}
It is not hard to see that for $1\leq i<j\leq n$,
\begin{equation}\label{CoProd}
\begin{aligned}
\Delta E_{ij} &= E_{ij} \otimes 1 + q^{E_{ii}-E_{jj}} \otimes E_{ij}+(q - q^{-1})\left [ \sum_{r=i+1}^{j-1} (q^{E_{rr}-E_{jj}} E_{ir}) \otimes E_{rj}  \right ] \\
\Delta E_{ji} &=1 \otimes E_{ji}  + E_{ji} \otimes q^{E_{jj}-E_{ii}} + (q - q^{-1}) \left [ \sum_{r=i+1}^{j-1} E_{ri} \otimes (q^{E_{rr}-E_{ii}}E_{jr})\right ].
\end{aligned}
\end{equation}
Note that because $bC-a$ is also central for any complex numbers $a,b$, any $bH-a\cdot \mathrm{Id}$ will also produce a suitable Hamiltonian. 

Although it will not be used explicitly, it is worth noting that there is an automorphism $\omega$ of $\Uq$ such that $\omega(E_{i,i+1})=E_{i+1,i},\  \omega(E_{i+1,i})=E_{i,i+1},$ and $\omega(q^{E_{ii}})=q^{-E_{ii}}$ (see e.g. Lemma 4.6(a) of \cite{J}). From the expression for $\Delta$, it can be seen that $\omega$ simultaneously reverses the order of the class system and the direction of the asymmetry, and hence provides an algebraic explanation for the CP--symmetry used below.

\subsection{Ground State Transformation}\label{GST}
To construct a ground state, \cite{CGRS} uses a $q$--exponential series applied to certain creation operators. 
Define
$$
\exp_{q^2} (x) = \sum_{n=0}^{\infty} \frac{x^n}{\Q{n}^!}
$$
where
$$
\Q{n}^! = \Q{1} \cdots \Q{n}
$$
and by convention $\Q{0}^!=1$. When $q\rightarrow 1$, $\exp_{q^2}(x)$ converges to the usual exponential $e^x$. By Proposition 5.1 of \cite{CGRS} and the definition of the co--product,
\begin{multline*}
\exp_{q^2} \left( \Delta^{(L-1)}E_{i,i+1} \right) \\
=  \exp_{q^2}\left(E_{i,i+1}^1\right) \exp_{q^2}\left( (q^{E_{ii}})^1 E_{i,i+1}^2\right) \exp_{q^2}\left( (q^{E_{ii}})^1 (q^{E_{ii}})^2 E_{i,i+1}^3\right)\cdots \exp_{q^2} \left((q^{E_{ii}})^1 \cdots (q^{E_{ii}})^{L-1}E_{i,i+1}^L \right)
\end{multline*}
where as usual, $(q^{E_{ii}})^x$ and $E_{i,i+1}^x$ mean that $q^{E_{ii}}$ and $E_{i,i+1}$ are acting at the lattice site indexed by $x$. 
Now define the operator
$$
S = \exp_{q^2}\left( \Delta^{(L-1)} E_{12} \right) \cdots \exp_{q^2}\left( \Delta^{(L-1)} E_{n-1,n} \right).
$$
Define $g:=S\Omega$ to be $S$ applied to the vacuum vector $\Omega$. By construction, the coefficients of $g$ in the canonical basis are all positive. Because $HS=SH$, we have
$$
Hg=HS\Omega = SH\Omega.
$$
From the explicit form of the central element $\Omega$ is an eigenvector of $H$. By replacing $H$ with $H-\lambda \mathrm{Id}$ where $H\Omega=\lambda\Omega$, this shows that $g$ is a positive ground state. One can check that $\lambda$ is the constant
$$
q^{-2n+3} + q^{-2n+5} + \ldots + q^{-1} + q^{4j+1}, 
$$
but this will not be explicitly used.

\textbf{Remark.} Note that one could also use Theorem 5.1 of \cite{CGRS} to define an operator
$$
S^- = \exp_{q^{-2}}\left( \Delta^{(L-1)}E_{n,n-1} \right) \cdots \exp_{q^{-2}}\left( \Delta^{(L-1)} E_{21} \right)
$$
and apply it to the particle configuration consisting of only first--class particles. However, this will ultimately result in the same process with the same duality. This can either be seen by a (long) calculation, or immediately with algebra. Because the sub--representation generated by the vacuum vector is the spin $jL$ representation, which has one--dimensional weight spaces, the ground state transformation is unchanged. Furthermore, there is an anti--automorphism of $\mathcal{U}_q(\mathfrak{g})$ which maps $E_{i,i+1}\mapsto E_{i+1,i}, E_{i+1,i}\mapsto E_{i,i+1},q^{E_{ii}}\mapsto q^{E_{ii}},$ (in the notation of 4.6 of \cite{J}, it is $\omega\circ\tau)$ so the duality is the same (up to a constant). Indeed, \cite{CGRS} carries through the calculation with both symmetries and arrives at the same duality function.

\subsection{Change of Basis}
In \cite{CGRS}, it is assumed that the Hamiltonian is self--adjoint. Here, this assumption is weakened slightly. Suppose that $B$ is a diagonal matrix such that $B^{-1}HB$ is self--adjoint. If $\mathcal{L} = G^{-1}HG$  then 
$$
\mathcal{L}^* = (G^{-1}B \cdot B^{-1}HB \cdot B^{-1}G)^* = B^{-1}G \cdot B^{-1}HB \cdot BG^{-1} = B^{-2} GHG^{-1} B^2.
$$
So if $D=G^{-1}SG^{-1}B^2$ where $SH=HS$ then
$$
\mathcal{L}D = G^{-1}HG G^{-1}SG^{-1}B^2 = G^{-1} SHG^{-1}B^2 = G^{-1}SG^{-1} B^2 \cdot G B^{-2}HB^2 G^{-1} = D\mathcal{L}^*
$$
and thus $D$ is a self--duality. 

In \cite{CGRS}, it is shown that $G^2$ defines a reversible measure assuming that $H$ is self--adjoint. It is straightforward to extend this to see that $G^2B^{-2}$ defines a reversible measure with the weakened assumptions. Indeed, the detailed balance condition reads
$$
G^2(x)B^{-2}(x) \cdot G^{-1}(x)H(x,y)G(y) = G^2(y)B^{-2}(y) \cdot G^{-1}(y)H(y,x)G(x),
$$
which is equivalent to 
$$
B^{-2}(x)H(x,y) = B^{-2}(y)H(y,x),
$$
which is itself equivalent to the assumption that $B^{-1}HB$ is self--adjoint.

The weakening of the self--adjointness assumption is chosen merely for reasons of convenience. Choosing a representation in which the Hamiltonian is self--adjoint could result in messier calculations, and it is easier to choose a simpler representation at the beginning and then multiply by $B^2$ at the very end. 

\subsection{Conjugation by Involution}\label{Invo}
Suppose that there is a matrix $V$ such that $\mathcal{L} = V^{-1} \mathcal{\tilde{L}}V$ where $\mathcal{\tilde{L}}$ is the generator for some continuous--time Markov process on the same state space as $\mathcal{L}$. Then if $D$ is a self--duality of $\mathcal{L}$,
$$
\mathcal{L}D = D \mathcal{L}^*  \Longrightarrow V^{-1}\mathcal{\tilde{L}}VD = DV^*\mathcal{\tilde{L}}^*(V^{-1})^* \Longrightarrow \mathcal{\tilde{L}}VDV^* = VDV^* \mathcal{\tilde{L}}^*
$$
so that $VDV^*$ is a self--duality of $\mathcal{L}^*$. Additionally,
$$
V^{-1}\tilde{\mathcal{L}}VD = D\mathcal{L}^* \Longrightarrow \tilde{\mathcal{L}}VD = VD\mathcal{L}^*
$$
so $VD$ is a duality between $\tilde{\mathcal{L}}$ and $\mathcal{L}$. 

In particular, if $V$ is a permutation matrix (that is, a matrix with exactly one $1$ in each row and column and zeroes everywhere else) such that $V=V^{-1}$, then $\mathcal{\tilde{L}}$ is automatically a generator, because the diagonal entries are still non--negative and the off--diagonal entries are still non--positive. If $T$ is an involution on the state space (i.e. a bijection such that $T=T^{-1}$), then taking $V$ to be the matrix of $T$ will result in a permutation matrix.

\section{Construction of Hamiltonian and Generator}\label{Constr}
 
\subsection{Hamiltonian}

Here, we compute an explicit formula for the Hamiltonian. Because $H$ is a sum of local operators $h^{x,x+1}$, it suffices to compute each $h^{x,x+1}$. Additionally, it is only necessary in the end to find the off--diagonal terms, because in the generator the sum of each row is $0$. Write $(\mu,\lambda)$ for the particle configurations at $x$ and $x+1$. 

\begin{proposition}\label{H} For $1\leq i<j\leq n$, 
\begin{align*}
h^{x,x+1}((\mu_{i \rightarrow j},\lambda_{j\rightarrow i}),(\mu,\lambda)) &= q^{-2} q^{\mu_i+\mu_{i+1}+\ldots+\mu_{j-1}}\Q{\mu_i} q^{\lambda_{i+1}+\ldots+\lambda_{j}} \Q{\lambda_j} \cdot q^{2(\lambda_{j+1}+\ldots+\lambda_n)} \cdot q^{2(\mu_1+\ldots+\mu_{i-1})},\\
h^{x,x+1}((\mu_{j \rightarrow i},\lambda_{i \rightarrow j}),(\mu,\lambda)) &= q^{\mu_{i}+\ldots+\mu_{j-1}}\Q{\mu_j} q^{\lambda_{i+1}+\ldots+\lambda_{j}} \Q{\lambda_i}  \cdot  q^{2(\lambda_{j+1}+\ldots+\lambda_n)} \cdot q^{2(\mu_1+\ldots+\mu_{i-1})}.
\end{align*}

\end{proposition}
\begin{proof}
Recall that each $h^{x,x+1}$ is defined from $\Delta(C)$ where $C$ is the central element from \eqref{GZBC}. Because the matrix elements of $h^{x,x+1}$ are defined with respect to the basis $\{v_{\mu} \otimes v_{\lambda}\}$, the proposition will be shown by finding the coefficients of $\Delta(C) \cdot (v_{\mu} \otimes v_{\lambda})$ with respect to the basis $\{v_{\mu} \otimes v_{\lambda}\}$.

When applying $\Delta(C)$, there will be many terms, which need to be grouped appropriately in order to find these coefficients. First do some bookkeeping to make sure that all the terms are accounted for. Applying \eqref{CoProd} to each $E_{ij}E_{ji}$ term in \eqref{GZBC} results in $(j-i+1)^2$ terms. However, terms that contribute to the diagonal entries, such as $E_{ij}E_{ji} \otimes 1$ will not be needed. Therefore the total number of terms relevant to computing the off--diagonal entries is 
$$
\sum_{1 \leq i < j \leq n} \left( (j-i+1)^2 - (j-i+1) \right) = \sum_{1 \leq i < j \leq n} \Big( 2(i-1)(n-j) + 2 + 2(n-j) + 2(i-1) \Big).
$$
With the terms grouped in the second sum, the sum counts the terms of the form (where $1 \leq r<i<j<s \leq n$)
$$
E_{ri}E_{jr} \otimes E_{is}E_{sj}, \quad E_{ij}\otimes E_{ji}, \quad E_{ij} \otimes E_{js}E_{si},\quad E_{ji}\otimes E_{rj} E_{ir},
$$
where these expressions come from $\Delta(E_{rs}E_{sr}), \Delta(E_{ij}E_{ji}), \Delta(E_{is}E_{si}),\Delta(E_{ri}E_{ir})$, respectively. With this grouping, all the terms in each group will contribute to the same coefficient.

So in order to calculate the first line of the proposition, using $q^{2n+2}$ times the central element from \eqref{GZBC}, apply the expression 
\begin{multline*}
q^{2j}q^{2E_{ii}}E_{ji}\otimes E_{ij}q^{2E_{jj}} + (q-q^{-1})\sum_{s=j+1}^n q^{2s}q^{2E_{ii}}E_{ji} \otimes E_{is}q^{E_{jj}+E_{ss}}E_{sj} \\
+ (q-q^{-1})q^{2j}\sum_{r=1}^{i-1} q^{E_{ii}+E_{rr}}E_{ri}E_{jr} \otimes q^{E_{jj}+E_{rr}}E_{ij}q^{E_{jj}-E_{rr}} \\
+ (q-q^{-1})^2\sum_{r=1}^{i-1}\sum_{s=j+1}^n q^{2s}q^{E_{ii}+E_{rr}}E_{ri}E_{jr} \otimes q^{E_{rr}+E_{ss}}E_{is}q^{E_{jj}-E_{rr}}E_{sj}
\end{multline*}
to $v_{\mu}\otimes v_{\lambda}$.
By applying $q^{E_{ii}}E_{ij}=q E_{ij}q^{E_{ii}}$ and $q^{E_{jj}}E_{ij} = q^{-1}E_{ij}q^{E_{jj}}$, this expression then equals 
\begin{multline*}
q^{2j+2}q^{2E_{ii}}E_{ji}\otimes q^{2E_{jj}} E_{ij}+ (q-q^{-1})\sum_{s=j+1}^n q^{2s+1}q^{2E_{ii}}E_{ji} \otimes q^{E_{ss}} E_{is}q^{E_{jj}}E_{sj} \\
+ (q-q^{-1})q^{2j+2}\sum_{r=1}^{i-1} q^{E_{ii}}E_{ri} q^{E_{rr}}E_{jr} \otimes q^{2E_{jj}}E_{ij}
+ (q-q^{-1})^2\sum_{r=1}^{i-1}\sum_{s=j+1}^n q^{2s+1}q^{E_{ii}}E_{ri}q^{E_{rr}}E_{jr} \otimes q^{E_{ss}}E_{is}q^{E_{jj}}E_{sj}.
\end{multline*}
When applied to $v_{\mu}\otimes v_{\lambda}$, by \eqref{LusztigAction2} the result is $v_{\mu_{i \rightarrow j}} \otimes v_{\lambda_{j \rightarrow i}}$ with coefficient
\begin{align*}
&q^{2j}  q^{\mu_{i}+\ldots+\mu_{j-1}}\Q{\mu_i}\cdot q^{\lambda_{i+1}+\ldots+\lambda_{j}}\Q{\lambda_j}\\
&+ (q-q^{-1})\sum_{s=j+1}^n q^{2s}q^{\mu_{i}+\ldots+\mu_{j-1}}\Q{\mu_i}  \cdot q^{\lambda_{i+1}+\ldots+(\lambda_j-1)+\ldots+\lambda_{s-1}}\Q{\lambda_s+1}  q^{\lambda_{j+1}+\ldots+\lambda_{s-1}}\Q{\lambda_j}\\
&+ (q-q^{-1})q^{2j+1}\sum_{r=1}^{i-1}  q^{\mu_{r+1}+\ldots+\mu_{i-1}}\Q{\mu_i} q^{\mu_{r+1}+\ldots+\mu_{j-1}} \Q{\mu_r} \cdot q^{\lambda_{i+1}+\ldots+\lambda_{j}}\Q{\lambda_j} \\
&+ (q-q^{-1})^2\sum_{r=1}^{i-1}\sum_{s=j+1}^n q^{2s+1} q^{\mu_{r+1}+\ldots+\mu_{i-1}} \Q{\mu_i} q^{\mu_{r+1}+\ldots+\mu_{j-1}}\Q{\mu_r} \\
& \quad \quad \quad \times  q^{\lambda_{i+1}+\ldots+(\lambda_j-1)+\ldots+\lambda_{s-1}}\Q{\lambda_s+1} q^{\lambda_{j+1}+\ldots+\lambda_{s-1}}\Q{\lambda_j}.
\end{align*}
Note that the $(\lambda_j-1)$ occurs in the second and fourth lines because $E_{sj}$ reduces $\lambda_j$ by $1$. Now factor the entire coefficient as 
\begin{multline*}
q^{\mu_i+\mu_{i+1}+\ldots+\mu_{j-1}}\Q{\mu_i} q^{\lambda_{i+1}+\ldots+\lambda_{j}} \Q{\lambda_j} \\
\times \left(q^{2j}+(q-q^{-1})\sum_{s=j+1}^n q^{2s-1} q^{2(\lambda_{j+1}+\ldots+\lambda_{s-1})} \Q{\lambda_s+1} \right)\left(1+(q-q^{-1}) \sum_{r=1}^{i-1} q\cdot q^{2(\mu_{r+1}+\ldots+\mu_{i-1})}\Q{\mu_r}\right).
\end{multline*}
The expression in the first parentheses can be simplified by using \eqref{Telescope2}, and the second parentheses can be simplified by using \eqref{Telescope}.
This therefore yields in total
$$
q^{\mu_i+\mu_{i+1}+\ldots+\mu_{j-1}}\Q{\mu_i} q^{\lambda_{i+1}+\ldots+\lambda_{j}} \Q{\lambda_j} \cdot q^{2n}q^{2(\lambda_{j+1}+\ldots+\lambda_n)} \cdot q^{2(\mu_1+\ldots+\mu_{i-1})}.
$$
Re--inserting the $q^{-2n-2}$ completes the first line in the Proposition.

Now for the second line. For $ r<i<j<s $, by similar reasoning, the grouping yields the expression 
\begin{multline*}
q^{2j}q^{E_{ii}+E_{jj}} E_{ij}\otimes q^{E_{ii}+E_{jj}} E_{ji} + (q-q^{-1})\sum_{s=j+1}^n q^{2s} q^{E_{ii} +E_{jj}} E_{ij} \otimes q^{E_{ii}+E_{ss}} E_{js}E_{si} \\
+ (q-q^{-1})q^{2j}\sum_{r=1}^{i-1}  q^{E_{rr}+E_{jj}}E_{rj} E_{ir} \otimes q^{E_{rr+}E_{jj}} q^{E_{ii}-E_{rr}}E_{ji} \\
+ (q-q^{-1})^2\sum_{r=1}^{i-1}\sum_{s=j+1}^n q^{2s} q^{E_{rr+}E_{ss}} q^{E_{jj}-E_{ss}}E_{rj}E_{ir} \otimes  q^{E_{rr+}E_{ss}} E_{js} q^{E_{ii}-E_{rr}}E_{si}.
\end{multline*}
Again by applying $q^{E_{ii}}E_{ij}=q E_{ij}q^{E_{ii}}$ and $q^{E_{jj}}E_{ij} = q^{-1}E_{ij}q^{E_{jj}}$, this then equals 
\begin{multline*}
q^{2j}q^{E_{ii}+E_{jj}} E_{ij}\otimes q^{E_{ii}+E_{jj}} E_{ji} + (q-q^{-1})\sum_{s=j+1}^n q^{2s} q^{E_{ii} +E_{jj}} E_{ij} \otimes q^{E_{ss}} E_{js} q^{E_{ii}}E_{si} \\
+ (q-q^{-1})q^{2j+1}\sum_{r=1}^{i-1}  q^{E_{jj}}E_{rj} q^{E_{rr}}E_{ir} \otimes q^{E_{ii}+E_{jj}}E_{ji} \\
+ (q-q^{-1})^2\sum_{r=1}^{i-1}\sum_{s=j+1}^n q^{2s+1} q^{E_{jj}} E_{rj} q^{E_{rr}}E_{ir} \otimes  q^{E_{ss}} E_{js} q^{E_{ii}} E_{si}.
\end{multline*}
Again by \eqref{LusztigAction2}, the coefficient of $v_{\mu_{j\rightarrow i}} \otimes v_{\lambda_{i \rightarrow j}}$ when applied to $v_{\mu} \otimes v_{\lambda}$ is
\begin{align*}
&q^{2j+2} q^{\mu_{i}+\ldots+\mu_{j-1}}\Q{\mu_j} q^{\lambda_{i+1}+\ldots+\lambda_{j}} \Q{\lambda_i} \\
&+ (q-q^{-1})\sum_{s=j+1}^n q^{2s+1} q^{\mu_i + \ldots + \mu_{j-1}}\Q{\mu_j} \cdot q^{\lambda_{j+1}+\ldots+\lambda_{s-1}}\Q{\lambda_s+1} q^{\lambda_{i+1}+\ldots+\lambda_{s-1}} \Q{\lambda_i} \\
&+(q-q^{-1}) q^{2j+2} \sum_{r=1}^{i-1} q^{\mu_{r+1}+\ldots+(\mu_{i}+1)+\ldots+\mu_{j-1}}\Q{\mu_j} q^{\mu_{r+1}+\ldots+\mu_{i-1}}\Q{\mu_i} q^{\lambda_{i+1}+\ldots+\lambda_j}\Q{\lambda_i}\\
&+ (q-q^{-1})^2\sum_{r=1}^{i-1}\sum_{s=j+1}^n q^{2s+1} q^{\mu_{r+1}+\ldots+(\mu_i+1)+\ldots+\mu_{j-1}} \Q{\mu_j} q^{\mu_{r+1}+\ldots+\mu_{i-1}} \Q{\mu_r} \\
& \quad \quad \quad \times q^{\lambda_{j+1}+\ldots+\lambda_{s-1}} \Q{\lambda_s+1} q^{\lambda_{i+1}+\ldots+\lambda_{s-1}} \Q{\lambda_i},
\end{align*}
which factors as 
\begin{multline*}
q^{2}q^{\mu_{i}+\ldots+\mu_{j-1}}\Q{\mu_j} q^{\lambda_{i+1}+\ldots+\lambda_{j}} \Q{\lambda_i} \\
\times \left( q^{2j} + (q-q^{-1}) \sum_{s=j+1}^n q^{2s-1} q^{2(\lambda_{j+1}+\ldots+\lambda_{s-1})} \Q{\lambda_s+1} \right)\left( 1 + (q-q^{-1})\sum_{r=1}^{i-1} q \cdot q^{2(\mu_{r+1}+\ldots+\mu_{i-1})}\Q{\mu_r}\right),
\end{multline*}
which equals 
$$
q^{2n+2}h^{x,x+1}((\mu_{j\rightarrow i},\lambda_{i\rightarrow j}),(\mu,\lambda)) = q^{2}q^{\mu_{i}+\ldots+\mu_{j-1}}\Q{\mu_j} q^{\lambda_{i+1}+\ldots+\lambda_{j}} \Q{\lambda_i}  \cdot q^{2n}q^{2(\lambda_{j+1}+\ldots+\lambda_n)} \cdot q^{2(\mu_1+\ldots+\mu_{i-1})}.
$$
This completes the proof.
\end{proof}
If one defines $B$ to be the diagonal matrix with diagonal entries
\begin{equation}\label{B}
B(\xi,\xi)= \prod_{x=1}^L \left(\sqrt{\Q{\xi^x_1}^!\cdots \Q{\xi^x_n}^!} \right)^{-1},
\end{equation}
then the conjugated Hamiltonian $B^{-1}HB$ is self--adjoint. Indeed, the entries are 
\begin{multline*}
B^{-1}HB((\mu_{i\rightarrow j},\lambda_{j\rightarrow i}),(\mu,\lambda)) \\
= q^{\mu_i+\mu_{i+1}+\ldots+\mu_{j-1}} \sqrt{ \Q{\mu_j+1}\Q{\mu_i} } \cdot q^{\lambda_{i+1}+\ldots+\lambda_{j}} \sqrt{ \Q{\lambda_i+1}  \Q{\lambda_j} }\cdot q^{2(\lambda_{j+1}+\ldots+\lambda_n)} \cdot q^{2(\mu_1+\ldots+\mu_{i-1})}
\end{multline*}
and 
\begin{multline*}
B^{-1}HB((\mu_{j\rightarrow i},\lambda_{i\rightarrow j}),(\mu,\lambda)) \\
= q^{2}q^{\mu_{i}+\ldots+\mu_{j-1}} \sqrt{ \Q{\mu_i+1}\Q{\mu_j} }  \cdot q^{\lambda_{i+1}+\ldots+\lambda_{j}} \sqrt{ \Q{\lambda_j+1}\Q{\lambda_i} } \cdot  q^{2(\lambda_{j+1}+\ldots+\lambda_n)} \cdot q^{2(\mu_1+\ldots+\mu_{i-1})}.
\end{multline*}
Making the replacement $\mu\mapsto \mu_{i\rightarrow j}$ and $\lambda \mapsto \lambda_{j\rightarrow i}$ in the second line, it is straightforward to see that the entries are equal. 

\subsection{Ground state transformation from Hamiltonian to Generator}
A ground state transformation will be applied to the quantum Hamiltonian in order to obtain a Markov generator.  Recall the definitions of $S$ and $g$ from Section \ref{GST}, and that $G$ is the diagonal matrix whose entries are the coefficients of $g$. 

For the next proofs, a bit more notation is needed. Note that due to the closed boundary conditions, the quantities
$$
M_i(\eta) = \eta_i^1 + \ldots + \eta_i^L
$$
representing the total number of $i$th class particles are conserved. As before, for $1\leq a \leq b \leq n$, let
$$
M_{[a,b]} (\eta)= M_a (\eta)+ \ldots + M_b (\eta).
$$
Note that if $D(\eta,\xi)$ is a duality function, and $c(\eta,\xi)$ is a conserved quantity under the dynamics, then $c(\eta,\xi)D(\eta,\xi)$ is also a duality.  The next identities concerning conserved quantities will be used in the proofs.

\begin{equation}\label{CQ1}
(M_{[1,i]}(\eta) - M_{[1,i]}(\xi))^2 = \left(\sum_{x=1}^L [ \eta_{[1,i]}^x - \xi_{[1,i]}^x ] \right)^2 = \sum_{x=1}^L ( \eta_{[1,i]}^x - \xi_{[1,i]}^x )^2  + 2 \sum_{1 \leq y < x \leq L}  ( \eta_{[1,i]}^y - \xi_{[1,i]}^y )( \eta_{[1,i]}^x - \xi_{[1,i]}^x )
\end{equation}
\begin{equation}\label{CQ2}
4j^2= \left( \sum_{i=1}^n \xi_i^x \right)^2 = \sum_{i=1}^n \left[ (\xi_i^x)^2 + 2\xi_i^x \xi_n^x \right] + 2\sum_{1\leq i < j \leq n} \xi_i^x \xi_j^x
\end{equation}
\begin{equation}\label{CQ3}
\sum_{1 \leq i < j \leq n-1} M_i(\xi)M_j(\xi) = \sum_{x=1}^L \sum_{z=1}^L \left( \sum_{1 \leq i < j \leq n-1} \xi_i^x \xi_j^z   \right)  = \sum_{x=1}^L \left( \sum_{1 \leq i < j \leq n-1} \xi_i^x \xi_j^x  + \sum_{z=x+1}^L \sum_{1 \leq i \neq j \leq n-1} \xi_i^x \xi_j^z   \right) 
\end{equation}

\begin{proposition}\label{S}  The expression $S(\eta,\xi)$ equals

$$
\mathrm{const} \cdot \prod_{x=1}^L  \prod_{i=1}^{n-1} 1_{\eta^x_{[1,i]} \geq \xi^x_{[1,i]}} \binomq{\eta_{[1,i+1]}^x - \xi_{[1,i]}^x}{ \eta_{i+1}^x} 
\times \prod_{1 \leq y < x \leq L} \prod_{i=1}^{n-1} \left( q^{\xi_i^y -  \eta_{i+1}^y  }  \right)^{  \eta_{[1,i]}^x - \xi_{[1,i]}^x  }  .
$$
where $\mathrm{const}$ is a conserved quantity under the dynamics.

In particular, $G(\eta)$ equals 
$$
\mathrm{const} \cdot \prod_{i=1}^{n} \prod_{x=1}^L \frac{1}{ \q{\eta_i^x}^!}
 \times \prod_{1 \leq y < x \leq L} \prod_{i=1}^{n-1}  q^{ - \eta_{i+1}^y  \eta_{[1,i]}^x } .
$$
\end{proposition}
\begin{proof}
In the definition of $S$, there are contributions from terms of the form $E_{i,i+1}$ and contributions from the terms of the form $q^{E_{ii}-E_{i+1,i+1}}$. The first step is to calculate the contributions from the $E_{i,i+1}$ terms. For each $x$ and $i$, apply \eqref{Binomial} with 
$$
m=(\eta_i^x - \xi_i^x) + \ldots + (\eta_1^x - \xi_1^x)
$$ and 
$$
\mu=(\xi_1^x , \ldots,  \xi_i^x  ,  \eta_{i+1}^x + (\eta_i^x - \xi_i^x) + \ldots + (\eta_1^x - \xi_1^x), \eta_{i+2}^x, \ldots, \eta_n^x   )
$$
to get 
$$
  q^{- \left( \eta_{[1,i]}^x - \xi_{[1,i]}^x \right) \left( \eta_{[1,i]}^x - \xi_{[1,i]}^x -1 \right) /2}   \binomq{\eta_{i+1}^x + \eta_{[1,i]}^x - \xi_{[1,i]}^x  }{\eta_{i+1}^x}.
$$
This leaves
$$
 \prod_{i=1}^{n-1}  q^{(M_{[1,i]}(\eta) - M_{[1,i]}(\xi))/2}  \prod_{x=1}^L \binomq{\eta_{[1,i+1]}^x - \xi_{[1,i]}^x}{ \eta_{i+1}^x}   \times \prod_{x=1}^L \prod_{i=1}^{n-1}  q^{-( \eta_{[1,i]}^x - \xi_{[1,i]}^x)^2/2 } .
$$
And now the contributions from the $q^{E_{ii}-E_{i+1,i+1}}$ terms are  
$$
\prod_{1 \leq y < x \leq L} \prod_{i=1}^{n-1} \left( q^{\xi_i^y - (\eta_{i+1}^y + (\eta_1^y - \xi_1^y) + \ldots + (\eta_i^y - \xi_i^y))} \right)^{(\eta_1^x - \xi_1^x) + \ldots + (\eta_i^x - \xi_i^x)}.
$$
Multiplying the two sets of contributions together, and using \eqref{CQ1}, yields the expression for $S(\eta,\xi)$.

Now by plugging in $\xi=\emptyset$,
$$
G(\eta)=C(\eta,\emptyset) \prod_{x=1}^L  \prod_{i=1}^{n-1}  \binomq{\eta_{[1,i+1]}^x  }{ \eta_{i+1}^x}  \times \prod_{x=1}^L \prod_{i=1}^{n-1}  q^{-( \eta_{[1,i]}^x )^2/2 }  
\times \prod_{1 \leq y < x \leq L} \prod_{i=1}^{n-1} \left( q^{ - \eta_{i+1}^y } \right)^{\eta_{[1,i]}^x }.
$$
Note that the binomial terms are a telescoping product
$$
\prod_{i=1}^{n-1} \binomq{\eta_{[1,i+1]}^x}{\eta_{i+1}^x} = \prod_{i=1}^{n-1} \frac{   \q{\eta_{[1,i+1]}^x}^!  }{    \q{\eta_{[1,i]}^x}^! \q{\eta_{i+1}^x}^!   }  =  \q{2j}^! \prod_{i=1}^{n} \frac{1}{ \q{\eta_i^x}^!}.
$$
\end{proof}

Note that setting $n=2$ yields a different expression than Lemma 5.5 of \cite{CGRS} because of the differing basis, due to this Hamiltonian not being self--adjoint.

\begin{proposition}\label{Gen} The generator $G^{-1}HG$ can be written as
$$
\mathcal{L} = \sum_{x=1}^L \mathcal{L}^{x,x+1},
$$
where for $1 \leq i < j \leq n$, 
\begin{align*}
\mathcal{L}^{x,x+1}((\mu,\lambda),(\mu_{j\rightarrow i},\lambda_{i\rightarrow j})) &= q \cdot q^{2(\mu_1+\ldots+\mu_{j-1})} \Q{\mu_j} \cdot q^{2(\lambda_{i+1} + \ldots + \lambda_n)} \Q{\lambda_i},\\
\mathcal{L}^{x,x+1}((\mu,\lambda),(\mu_{i\rightarrow j},\lambda_{j\rightarrow i})) &= q^{-1} \cdot q^{2(\mu_1+\ldots+\mu_{i-1})}\Q{\mu_i} \cdot q^{2(\lambda_{j+1}+\ldots+\lambda_n)} \Q{\lambda_j}.\\    
\end{align*}
\end{proposition}
\begin{proof}
Recall that by Propositions \ref{H} and \ref{S},
$$
h^{x,x+1}((\mu_{i\rightarrow j},\lambda_{j\rightarrow i}),(\mu,\lambda)) = q^{-2}q^{\mu_i+\mu_{i+1}+\ldots+\mu_{j-1}}\Q{\mu_i} q^{\lambda_{i+1}+\ldots+\lambda_{j}} \Q{\lambda_j} \cdot q^{2(\lambda_{j+1}+\ldots+\lambda_n)} \cdot q^{2(\mu_1+\ldots+\mu_{i-1})}
$$ 
and
\begin{align*}
\frac{ G(\mu,\lambda) }{ G(\mu_{i\rightarrow j},\lambda_{j\rightarrow i}) } &= \frac{\q{\mu_j+1}}{\q{\mu_i}}   \frac{\q{\lambda_i+1}}{\q{\lambda_j}} \times q^{ -\mu_i\lambda_{[1,i-1]} - \mu_{i+1} \lambda_{[1,i]} - \ldots - \mu_j\lambda_{[1,j-1]   } }\\
& \quad \quad  \times q^{ (\mu_i-1)\lambda_{[1,i-1]} + \mu_{i+1}(\lambda_{[1,i]}+1) + \ldots + \mu_{j-1}(\lambda_{[1,j-2]}+1)+(\mu_j+1)(\lambda_{[1,j-1]}+1)    }\\
&= \frac{\q{\mu_j+1}}{\q{\mu_i}}   \frac{\q{\lambda_i+1}}{\q{\lambda_j}}   q^{-\lambda_{[1,i-1]} + \mu_{i+1} + \ldots + \mu_{j} + \lambda_{[1,j-1]}+1},
\end{align*}
which implies that $\mathcal{L}^{x,x+1}((\mu_{i\rightarrow j},\lambda_{j\rightarrow i}),(\mu,\lambda))$ equals
$$
 q^{-3} q^{2(\mu_i+\mu_{i+1}+\ldots+\mu_{j-1})}q^{\mu_j}\q{\mu_j+1} q^{2(\lambda_{i+1}+\ldots+\lambda_{j-1})} q^{\lambda_i}\q{\lambda_i + 1} \cdot q^{2(\lambda_{j}+\ldots+\lambda_n)} \cdot q^{2(\mu_1+\ldots+\mu_{i-1})}.
$$
Replacing $\mu$ with $\mu_{j\rightarrow i}$ and $\lambda$ with $\lambda_{i\rightarrow j}$ finishes showing the first line. An identical argument holds for the second.
\end{proof}

\subsection{Markov Projections}
Recall the projections $\Pi^{\sigma}, \Pi^n_{i+1}$ and $\tilde{\Pi}^n_{i+1}$ defined in Section \ref{Desc}.

\begin{proposition}\label{MarkovProjection}
The projection of type $A_n$, spin $j$ ASEP under $\Pi^n_{i+1}$ and $\tilde{\Pi}^n_{i+1}$ is a type $A_{i+1}$, spin $j$ ASEP.  The projection of type $A_n$ $q$--TAZRP under $\Pi^n_{i+1}$ and $\tilde{\Pi}^n_{i+1}$ is a type $A_{i+1}$ $q$--TAZRP.

For any $\sigma:\{1,\ldots,n\} \rightarrow \{1,\ldots,i+1\}$, the projection of type $A_n$, spin $j$ SSEP under $\Pi_{\sigma}$ is a type $A_{i+1}$, spin $j$ SSEP.

\end{proposition}
\begin{proof}
Observe that $\Pi^{\sigma_1 \circ \sigma_2} = \Pi^{\sigma_1} \circ \Pi^{\sigma_2}$. In particular, this implies that $\Pi^k_{i+1} \circ \Pi^n_k = \Pi^n_{i+1}$ and $\tilde{\Pi}^k_{i+1} \circ \tilde{\Pi}^n_k = \tilde{\Pi}^n_{i+1}$ for $i+1< k < n$. Therefore it suffices to prove the \textrm{ASEP} and $q$--TAZRP statements for $i+1=n-1$.

Letting $\mathcal{L}_n$ denote the generator of the type $A_n$, spin $j$ ASEP, the proposition follows by showing that $\Pi_{n-1}^n \circ \mathcal{L}_n = \mathcal{L}_{n-1} \circ \Pi_{n-1}^n$. By Proposition \ref{Gen}, $\mathcal{L}_n = \sum_{x=1}^L \mathcal{L}_n^{x,x+1}$, so it suffices to show that
$$
\Pi_{n-1}^n \circ \mathcal{L}_n^{x,x+1} = \mathcal{L}_{n-1}^{x,x+1} \circ \Pi_{n-1}^n,
$$ 
Letting $(\mu,\lambda)$ denote the $n$--species particle configurations at $x$ and $x+1$ respectively, it suffices to check that both sides are equal at the matrix entry with row indexed by $(\mu,\lambda)$ and column indexed by $(\Pi^n_{n-1}\mu)_{l\rightarrow k}, (\Pi^n_{n-1}\lambda)_{k\rightarrow l}$. If both $k$ and $l$ are strictly less than $n-1$, then this is immediate, which can be seen either computationally, or because the projection only affects $n-1$ and $n$th class particles. For $l=n-1$, observe that the left--hand--side is
$$
q\cdot q^{2(\lambda_{k+1} + \ldots + \lambda_{n})} \Q{\lambda_k} q^{2(\mu_1+\ldots+\mu_{n-1})} \Q{\mu_n} + q \cdot q^{2(\lambda_{k+1} + \ldots + \lambda_{n})} \Q{\lambda_k} q^{2(\mu_1+\ldots+\mu_{n-2})} \Q{\mu_{n-1}} 
$$
and the right--hand--side is
$$
q\cdot q^{2(\lambda_{k+1} + \ldots + \lambda_{n})} \Q{\lambda_k} q^{2(\mu_1+\ldots+\mu_{n-2})} \Q{\mu_{n-1}+\mu_n}. 
$$
The two sides are equal because by \eqref{Projection}, 
$$
 \Q{\mu_{n-1}+\mu_n} = q^{2\mu_{n-1}}\Q{\mu_n} +  \Q{\mu_{n-1}}.
$$
For $k=n-1$, the argument is identical. This implies that $\Pi^n_{n-1}$ maps type $A_n$, spin $j$ ASEP to its type $A_{n-1}$ version. Because the equality still holds as $j\rightarrow\infty$, the statement also holds for type $A_n$ $q$--TAZRP.

For $\tilde{\Pi}^n_{n-1}$ a similar argument holds. The equality to check is
\begin{multline*}
q \cdot q^{2(  \lambda_{n})} \Q{\lambda_{n-2}+\lambda_{n-1}} q^{2(\mu_1+\ldots+(\mu_{n-2}+\mu_{n-1}))} \Q{\mu_n} \\
= q\cdot q^{2(\lambda_{n-1} + \lambda_{n})} \Q{\lambda_{n-2}} q^{2(\mu_1+\ldots+\mu_{n-1})} \Q{\mu_n} + q\cdot q^{2\lambda_{n}} \Q{\lambda_{n-1}} q^{2(\mu_1+\ldots+\mu_{n-1})} \Q{\mu_n} .
\end{multline*} 
Again, this is true because \eqref{Projection} implies 
$$
\Q{\lambda_{n-2}+\lambda_{n-1}} =  q^{2\lambda_{n-1}} \Q{\lambda_{n-2}} + \Q{\lambda_{n-1}}.
$$
Therefore, $\tilde{\Pi}^n_{n-1}$ maps type $A_n$, spin $j$ ASEP and $q$--TAZRP to its type $A_{n-1}$ version.

Now turn to the argument for SSEP. Without loss of generality, assume that $\sigma$ is surjective. Any surjection $\sigma:\{1,\ldots,n\} \rightarrow \{1,\ldots,i+1\}$ can be defined as a composition of surjections $\{1,\ldots,n\} \rightarrow \{1,\ldots,n-1\} \rightarrow \ldots \rightarrow \{1,\ldots,i+1\}$. Therefore it suffices to consider when $i+1=n-1$. In this case, for any surjection $\sigma: \{1,\ldots,n\} \rightarrow \{1,\ldots,n-1\}$, there is a $\tau \in S_{n-1}$ such that
$$
\Pi^{\sigma} = \Pi^{\tau} \circ \Pi^{n}_{n-1}.
$$
Since $\Pi^{\tau} \circ \mathcal{L}_{n-1}^{x,x+1} = \mathcal{L}_{n-1}^{x,x+1}  \circ \Pi^{\tau}$, 
$$
\Pi^{\sigma} \circ \mathcal{L}_n^{x,x+1} = \Pi^{\tau} \circ \Pi^{n}_{n-1}\circ \mathcal{L}_n^{x,x+1} = \Pi^{\tau} \circ  \mathcal{L}_{n-1}^{x,x+1} \circ \Pi_{n-1}^n = \mathcal{L}_{n-1}^{x,x+1}  \circ \Pi^{\tau}\circ \Pi_{n-1}^n  = \mathcal{L}_{n-1}^{x,x+1}  \circ  \circ \Pi^{\sigma},
$$
as needed.
\end{proof}


\subsection{Relationships to other processes}
 
\subsubsection{$n=2$ and $j$ arbitrary}
From \cite{CGRS}, equation (11), the right jump rates of $\textrm{ASEP}(q,j)$ are
$$
q^{-4j + 1} \Q{\mu_1} \Q{\lambda_2}
$$
and the left jump rates are
$$
 q^{2\mu_1 -2\lambda_1 +3} \Q{\mu_2} \Q{\lambda_1} =  q^{2\mu_1 + 2\lambda_2  -4j  + 3} \Q{\mu_2} \Q{\lambda_1}
$$
So that after a rescaling of time by $q^{4j-2}$, the process here is precisely $\textrm{ASEP}(q,j)$ with closed boundary conditions. Additionally, at $q= 1$ the $\textrm{ASEP}(q,j)$ is the $2j$--SEP from \cite{GKRV}.

\subsubsection{$j=1/2$ and $n$ arbitrary} 
In this case, only one particle is allowed at each site, so if $\mu_j=1$ then $q^{2(\mu_1+\ldots+\mu_{j-1})}=1$. Since $\Q{1}=1$, the process reduces to $(n-1)$--species ASEP with closed boundary conditions. 

\subsubsection{$j\rightarrow\infty$ limit}

In the infinite spin limit, we have $\lambda_n\rightarrow\infty$. Recall that $\Q{\infty} = (1-q^2)^{-1}$. Since $i<n$, the expression for the left jumps must contain a $q^{\lambda_n}$ term which goes to zero. The expression for the right jumps only has a nonzero limit if $j=n$. Thus the only jumps are for a particle of class $i$ jumping to the right at rates:
$$
(q^{-1}-q)^{-1} q^{2(\mu_1+\ldots+\mu_{i-1})}\Q{\mu_i} .
$$
This then becomes a $(n-1)$--species $q$--TAZRP from \cite{T}.

\section{Stochastic Duality and Reversible Measures}\label{Results}

The next proposition implies Theorem \ref{DualityTheorem}(a).

\begin{proposition}The expression $D(\eta,\xi) = G^{-1}(\eta) S(\eta,\xi) G^{-1}(\xi) B^2(\xi)$ equals

$$
\mathrm{const}\cdot
\prod_{x=1}^L \q{\eta_1^x}^!  \prod_{i=1}^{n-1}   1_{\eta_{[1,i]}^x \geq \xi_{[1,i]}^x}   \frac{  \q{  \eta_{[1,i+1]}^x - \xi_{[1,i]}^x   }^!  }{\q{\eta_{[1,i]}^x   - \xi_{[1,i]}^x }^! }  q^{4jx\xi_i^x} q^{ \xi_i^x (\sum_{z=x+1}^L 2\eta_{[1,i]}^z + \eta_{[1,i]}^x)}
$$
where $\mathrm{const}$ is a conserved quantity under the dynamics.

\end{proposition}
\begin{proof}
By Proposition \ref{S} and \eqref{B}, $D(\eta,\xi)$ equals (up to a constant)
$$
\prod_{x=1}^L \prod_{i=1}^n \frac{\q{\eta_i^x}^!\q{\xi_i^x}^!}{\Q{\xi_i^x}^!} \prod_{i=1}^{n-1} 1_{\eta_{[1,i]}^x \geq \xi_{[1,i]}^x} \binomq{ \eta_{[1,i+1]}^x - \xi_{[1,i]}^x   }{  \eta_{i+1}^x   } \prod_{1 \leq y < x \leq L} \prod_{i=1}^{n-1}  q^{  \xi_i^y\eta_{[1,i]}^x + \eta_{i+1}^y\xi_{[1,i]}^x  } q^{(\xi_{i+1}^y - \xi_i^y)\xi_{[1,i]}^x}.
$$
Now using that
$$
\prod_{x=1}^L \prod_{i=1}^n  \frac{\q{\xi_i^x}^!}{\Q{\xi_i^x}^!} = \prod_{x=1}^L \prod_{i=1}^n q^{-\xi_i^x(\xi_i^x-1)/2} = \prod_{i=1}^n q^{M_i(\xi)/2} \prod_{x=1}^L q^{- (\xi_i^x)^2/2}
$$
and
$$
\prod_{i=1}^n \q{\eta_i^x}^! \prod_{i=1}^{n-1} \frac{  \q{  \eta_{[1,i+1]}^x - \xi_{[1,i]}^x   }^!  }{ \q{\eta_{i+1}^x}^! \q{\eta_{[1,i]}^x   - \xi_{[1,i]}^x }^! }  =   \q{\eta_1^x}^! \prod_{i=1}^{n-1} \frac{  \q{  \eta_{[1,i+1]}^x - \xi_{[1,i]}^x   }^!  }{\q{\eta_{[1,i]}^x   - \xi_{[1,i]}^x }^! } 
$$
simplifies the expression for $D(\eta,\xi)$ (up to a constant) to 
$$
\prod_{x=1}^L  q^{-(\xi_1^x)^2/2}\prod_{i=1}^{n-1}   1_{\eta_{[1,i]}^x \geq \xi_{[1,i]}^x}   \frac{  \q{  \eta_{[1,i+1]}^x - \xi_{[1,i]}^x   }^!  }{\q{\eta_{[1,i]}^x   - \xi_{[1,i]}^x }^! }  q^{-(\xi^x_{i+1})^2/2}\prod_{1 \leq y < x \leq L} \prod_{i=1}^{n-1}  q^{  \xi_i^y\eta_{[1,i]}^x + \eta_{i+1}^y\xi_{[1,i]}^x  } q^{(\xi_{i+1}^y - \xi_i^y)\xi_{[1,i]}^x}.
$$
Now rewrite
\begin{align*}
\prod_{1\leq y<x\leq L} \prod_{i=1}^{n-1} q^{ (\xi_{i+1}^y - \xi_i^y)\xi_{[1,i]}^x } &= \prod_{1\leq y<x\leq L} \prod_{i=1}^{n-1} q^{ \left(\xi_{[i+1,n]}^y - \xi_{[i,n-1]}^y\right)\xi_i^x } \\
&= \prod_{i=1}^{n-1} \prod_{x=1}^L\prod_{y=1}^{x-1} q^{ (\xi_{n}^y - \xi_{i}^y)\xi_i^x } \\
&= \prod_{i=1}^{n-1}q^{M_i(\xi)M_n(\xi)}\prod_{x=1}^L q^{-\xi_i^x\xi_n^x}\prod_{z=x+1}^L q^{ -\xi_{i}^x(\xi_i^z + \xi_n^z )} .
\end{align*}
Using \eqref{CQ2}, the duality function becomes
$$
\prod_{x=1}^L  \q{\eta_1^x}^!  \prod_{i=1}^{n-1}   1_{\eta_{[1,i]}^x \geq \xi_{[1,i]}^x}   \frac{  \q{  \eta_{[1,i+1]}^x - \xi_{[1,i]}^x   }^!  }{\q{\eta_{[1,i]}^x   - \xi_{[1,i]}^x }^! }  \prod_{1 \leq i< j \leq n-1} q^{\xi_i^x \xi_j^x}\prod_{1 \leq y < x \leq L} \prod_{i=1}^{n-1}  q^{  \xi_i^y\eta_{[1,i]}^x + \eta_{i+1}^y\xi_{[1,i]}^x  } \prod_{z=x+1}^L q^{-\xi_i^x(\xi_i^z + \xi_n^z)}
$$
By plugging in  $\xi_i^z + \xi_n^z = 2j - \sum_{j\neq i} \xi_j^z$, and using \eqref{CQ3} the duality function becomes 
$$
\prod_{x=1}^L \q{\eta_1^x}^!  \prod_{i=1}^{n-1}   1_{\eta_{[1,i]}^x \geq \xi_{[1,i]}^x}   \frac{  \q{  \eta_{[1,i+1]}^x - \xi_{[1,i]}^x   }^!  }{\q{\eta_{[1,i]}^x   - \xi_{[1,i]}^x }^! }  q^{2jx\xi_i^x} \prod_{1 \leq y < x \leq L} \prod_{i=1}^{n-1}  q^{  \xi_i^y\eta_{[1,i]}^x + \eta_{i+1}^y\xi_{[1,i]}^x  }  .
$$
And the last product simplifies as
\begin{align*}
\prod_{1 \leq y < x \leq L} \prod_{i=1}^{n-1}  q^{  \xi_i^y\eta_{[1,i]}^x + \eta_{i+1}^y\xi_{[1,i]}^x  }  &=  \prod_{i=1}^{n-1}\prod_{1 \leq y < x \leq L}  q^{ \eta_{[i+1,n]}^y\xi_{i}^x  }  \prod_{1 \leq x < z \leq L}  q^{ \xi_i^x\eta_{[1,i]}^z } \\
&= \prod_{i=1}^{n-1} \prod_{x=1}^Lq^{M_{[i+1,n]}(\eta)M_i(\xi)} q^{-\xi_i^x\eta_{[i+1,n]}^x} \prod_{1 \leq x < z \leq L} q^{\xi_i^x ( \eta_{[1,i]}^z - \eta^z_{[i+1,n]} ) }\\
& = \mathrm{const} \cdot \prod_{i=1}^{n-1} \prod_{x=1}^L q^{ \sum_{z=x+1}^L 2\eta_{[1,i]}^z \xi_i^x + \eta_{[1,i]}^x\xi_i^x + 2jx\xi_i^x}.
\end{align*}
to get the result.

\end{proof}

\subsection{$j=1/2$}
Note that when $j=1/2$, the binomial terms $\binomq{1}{0}=\binomq{1}{1}=\binomq{0}{0}$ are always $1$. Therefore the duality function simplifies to
$$
D(\eta,\xi) = \prod_{i=1}^{n-1} \prod_{x=1}^L \prod_{z=x+1}^L q^{2\xi^x_i\eta^z_{[1,i]} +2\xi_i^xx}.
$$
In particular, when $n=3$ this is the duality from \cite{BS} and \cite{K}, and when $n=2$ this is the duality function from \cite{S}.

\subsection{$n=2$}
In this case, the duality function is
$$
D(\eta,\xi)=C\prod_{x=1}^L \frac{ \binomq{\eta^x}{\xi^x} }{ \binomq{2j}{\xi^x} }q^{-\xi^x(2j-\eta^x)  } \prod_{z=x+1}^L q^{2\xi^x(\eta^z - 2j)  } = C' \prod_{x=1}^L q^{\xi^x(2\sum_{z=x+1}^L\eta^z + \eta^x + 4jx)}.
$$
Compare this to the duality function from \cite{CGRS}
\begin{equation}\label{CGRSDuality}
\begin{aligned}
D(\eta,\xi) &= \prod_{x=1}^L  \frac{ \binomq{\eta^x}{\xi^x} }{ \binomq{2j}{\xi^x} } \cdot q^{(\eta^x-\xi^x)\left[ 2\cdot 1_{x\geq 2} \sum_{y=1}^{x-1} \xi^y + \xi^x \right] + 4jx\xi^x} 1_{\eta^x \geq \xi^x}, \\
D'(\eta,\xi) &= \prod_{x=1}^L  \frac{ \binomq{\eta^x}{\xi^x} }{ \binomq{2j}{\xi^x} } \cdot q^{(\eta^x-\xi^x)\left[ 2\cdot 1_{x\geq 2} \sum_{y=1}^{x-1} \eta^y + \eta^x \right] + 4jx\xi^x} 1_{\eta^x \geq \xi^x}.
\end{aligned}
\end{equation}
Since
$$
\sum_{x=1}^L \eta^x\left( 2\sum_{y=1}^{x-1} \eta^y + \eta^x\right) = \left( \sum_{x=1}^L \eta^x\right)^2
$$
the second function is
$$
D'(\eta,\xi) = q^{M_1(\eta)^2-2M_1(\eta)M_1(\xi)}\prod_{x=1}^L  \frac{ \binomq{\eta^x}{\xi^x} }{ \binomq{2j}{\xi^x} } \cdot q^{\xi^x\left[ 2 \sum_{z=x+1}^{L} \eta^z + \eta^x \right] + 4jx\xi^x} 1_{\eta^x \geq \xi^x}
$$
so the dualities match up to a constant. Note that this same argument implies
$$
D(\eta,\xi) = q^{-M_1(\xi)^2} \prod_{x=1}^L  \frac{ \binomq{\eta^x}{\xi^x} }{ \binomq{2j}{\xi^x} } \cdot q^{\eta^x\left[ 2\cdot 1_{x\geq 2} \sum_{y=1}^{x-1} \xi^y + \xi^x \right] + 4jx\xi^x} 1_{\eta^x \geq \xi^x} ,
$$
so in fact, $D$ and $D'$ are the same to up a constant.

Additionally, when $q=1$ the duality function $D(\eta,\xi)$ reduces to that of \cite{GKRV}.
 
\subsection{CP Symmetry}
Let $T$ be the function on the state space which reverses the class system of the particles. That is, 
$$
(T(\xi))^x_i = \xi^x_{n+1-i} \text{ for } 1 \leq x \leq L, \ 1 \leq i \leq n.
$$
Take $V(\eta,\xi) = 1_{\{\eta = T(\xi)\}}$ in Section \ref{Invo}. The next proposition implies Theorem \ref{DualityTheorem}(b).

\begin{proposition}\label{CPS}
(a) The generator of space--reversed $\mathrm{ASEP}(q,j)$ is $V^{-1}\mathcal{L}V$. 

(b)
The expression $D(T(\eta),\xi)$ equals
$$
\mathrm{const}\cdot\prod_{x=1}^L \q{2j-\eta_{[1,n-1]}^x}^!  \prod_{i=1}^{n-1}   1_{2j-\eta_{[1,n-i]}^x \geq \xi_{[1,i]}^x}   \frac{  \q{  2j - \eta_{[1,n-i-1]}^x - \xi_{[1,i]}^x   }^!  }{\q{2j - \eta_{[1,n-i]}^x   - \xi_{[1,i]}^x }^! }   q^{ -\xi_i^x (\sum_{z=x+1}^L 2\eta_{[1,n-i]}^z + \eta_{[1,n-i]}^x)} .
$$
In the $j\rightarrow\infty$ limit, $D(T(\eta),\xi)$ converges to 
$$
\mathrm{const}\cdot\prod_{x=1}^L \prod_{i=1}^{n-1} q^{ \xi_i^x (\sum_{y=1}^x 2\eta_{[1,n-i]}^z) } .
$$

\end{proposition}
\begin{proof} 
(a) 
Note that
$$
V^{-1}\mathcal{L}V = \sum_{x=1}^L \tilde{\mathcal{L}}^{x,x+1}
$$
where for any $i<j$,
\begin{align*}
\tilde{\mathcal{L}}^{x,x+1}\Big( (\mu,\lambda) \rightarrow (\mu_{j\rightarrow i},\lambda_{i\rightarrow j}) \Big) &= \mathcal{L}^{x,x+1}\Big((T(\mu),T(\lambda)) \rightarrow (T(\mu_{j\rightarrow i}),T(\lambda_{i\rightarrow j}))\Big)  \\
&= \mathcal{L}^{x,x+1}\Big( (T(\mu),T(\lambda)) \rightarrow T(\mu)_{n+1-j \rightarrow n+1-i}, T(\lambda)_{n+1-i \rightarrow n+1-j}   \Big)  \\
&= q^{-1} \cdot q^{2(T(\mu)_1 + \ldots + T(\mu)_{n-j})} \Q{T(\mu)_{n+1-j}} \cdot q^{2(T(\lambda)_{n+2-i} + \ldots + T(\lambda)_n )} \Q{T(\lambda)_{n+1-i}}\\
&= q^{-1} \cdot q^{2(\mu_n + \ldots + \mu_{j+1})} \Q{\mu_j} q^{2(\lambda_1 + \ldots + \lambda_{i-1})} \Q{\lambda_i} \\
&= \mathcal{L}^{x,x+1}( (\lambda,\mu) \rightarrow (\lambda_{i\rightarrow j},\mu_{j\rightarrow i}) ),
\end{align*}
and similarly
$$
\tilde{\mathcal{L}}_{x,x+1}\left( (\mu,\lambda) \rightarrow (\mu_{i\rightarrow j},\lambda_{j\rightarrow i}) \right) = \mathcal{L}_{x,x+1}( (\lambda,\mu) \rightarrow (\lambda_{j\rightarrow i},\mu_{i\rightarrow j}) )   .
$$
Therefore, $\tilde{\mathcal{L}}= V^{-1}\mathcal{L}V$ is the generator of space--reversed $\mathrm{ASEP}(q,j)$.

(b)
By the definition of $T$, $D(T(\eta),\xi)$ equals
$$
 \prod_{x=1}^L \q{\eta_n^x}^!  \prod_{i=1}^{n-1}   1_{\eta_{[n-i+1,n]}^x \geq \xi_{[1,i]}^x}   \frac{  \q{  \eta_{[n-i,n]}^x - \xi_{[1,i]}^x   }^!  }{\q{\eta_{[n-i+1,n]}^x   - \xi_{[1,i]}^x }^! }   q^{ \xi_i^x (\sum_{z=x+1}^L 2\eta_{[n-i+1,n]}^z + \eta_{[n-i+n]}^x)} q^{4jx\xi_i^x}.
$$
Using that $ \eta^x_1 + \ldots + \eta^x_n = 2j$  gives the expression in the Proposition.

As $j\rightarrow\infty$, the indicator function always holds. For the $q$--factorial terms, use that for a fixed finite $k>0$,
$$
\lim_{j\rightarrow \infty} (q^{-1}-q)^{2j-k} q^{(2j-k)(2j-k+1)/2}\q{2j-k}^! = \lim_{j\rightarrow\infty}\prod_{l=1}^{2j-k} q^{l}(q^{-l} - q^{l}) = c(q)
$$
where $c(q)= \prod_{l=1}^{\infty} (1-q^{2l})$. 
Therefore, the limit of the $q$--factorial terms is (up to the constant $c(q)$)
\begin{align*}
\prod_{x=1}^L q^{-(\eta^x_{[1,n-1]})^2/2} \prod_{i=1}^{n-1} \frac{q^{-( \eta_{[1,n-i-1]}^x + \xi^x_{[1,i]} )^2/2}}{q^{-(\eta_{[1,n-i]}^x + \xi^x_{[1,i]})^2/2}} &= \prod_{x=1}^L q^{-(\eta^x_{[1,n-1]})^2/2} \prod_{i=1}^{n-1} \frac{q^{-( \eta_{[1,n-i-1]}^x)^2/2 - \eta_{[1,n-i-1]}^x \xi^x_{[1,i]}   }}{q^{-( \eta_{[1,n-i]}^x)^2/2 - \eta_{[1,n-i]}^x \xi^x_{[1,i]} }}\\
&= \prod_{x=1}^L q^{-(\eta_{[1,n-1]}^x)^2/2} \prod_{i=1}^{n-1} q^{( \eta^x_{[1,n-i-1]} + \xi^x_{[1,i]})\eta_{n-i}^x + (\eta_{n-i}^x)^2/2}.
\end{align*}
And note that
$$
-\frac{(\eta^x_{[1,n-1]})^2}{2} = -\frac{1}{2}\left( \sum_{i=1}^{n-1} \eta_i^x \right)^2 = -\frac{1}{2}\sum_{i=1}^{n-1}  ( (\eta^x_i)^2 + 2\eta^x_{[1,i-1]}\eta^x_i   ).
$$
Since $\sum_{i=1}^{n-1} \xi^x_{[1,i]}\eta^x_{n-i} = \sum_{i=1}^{n-1}\xi^x_i \eta^x_{[1,n-i]}$ the duality function simplifies to 
$$
\prod_{x=1}^L \prod_{i=1}^{n-1} q^{-\xi_i^x\left( \sum_{z=x+1}^L 2\eta_{[1,n-i]}^z \right)} = \mathrm{const} \cdot \prod_{x=1}^L \prod_{i=1}^{n-1} q^{\xi_i^x\left( \sum_{y=1}^x 2\eta_{[1,n-i]}^z \right)},
$$
finishing the proof.

\end{proof}

Here is an example demonstrating this. 
Suppose $L=2,n=3$ and $\eta_1^1=1,\eta_2^1=0,\eta_1^2=2,\eta_2^2=3$ and $\xi^1_1=1,\xi^1_2=1,\xi^2_1=1,\xi^2_2=0$. Visually, (note that the holes represented by $3$'s are omitted, because there are infinitely many of them)
$$
\eta = 
\left( 
\begin{array}{cc}
 & 2 \\
 & 2 \\
 & 2 \\
 & 1 \\
1 & 1 
\end{array}
\right)
\quad \quad 
\xi=
\left(
\begin{array}{cc}
2 & \\
1 & 1
\end{array}
\right)
$$
Then we can calculate
\begin{align*}
\tilde{\mathcal{L}}D(\eta,\xi) &= -( \Q{2} + q^4\Q{3})q^{16} + \Q{2}q^{20} + q^4\Q{3}q^{18}\\
&= q^{-16}(-1-q^2) + q^{26} + q^2\cdot q^{20}\\
&= D\mathcal{L}^*(\eta,\xi)
\end{align*}

\subsection{Reversible Measures}
\begin{proposition} When restricted to states with a fixed number of particles, there is a unique reversible measure given by
$$
 \mathbb{P}(\xi) \propto \prod_{x=1}^L  \prod_{i=1}^{n}  \frac{1}{\q{\xi_i^x}^!} q^{ ( \xi_i^x)^2/2 } \prod_{1 \leq y < x \leq L} \prod_{i=1}^{n-1}  q^{ -   2 \xi_{i+1}^y \xi_{[1,i]}^x  } .
$$
\end{proposition}
\begin{proof} Using \eqref{B} this results in 
$$
G^2(\xi)B^{-2}(\xi) = G(\xi)^2 \prod_{x=1}^L \Q{\xi_1^x}^! \cdots \Q{\xi_n^x}^!
$$
and plugging in $G(\xi)$ from Proposition \ref{S} yields
$$
C \prod_{x=1}^L  \prod_{i=1}^{n}  \frac{\Q{\xi_i^x}^!}{(\q{\xi_i^x}^!)^2}  \prod_{1 \leq y < x \leq L} \prod_{i=1}^{n-1}  q^{ -   2 \xi_{i+1}^y \xi_{[1,i]}^x  } = C' \prod_{x=1}^L  \prod_{i=1}^{n}  \frac{1}{\q{\xi_i^x}^!} q^{ ( \xi_i^x)^2/2 } \prod_{1 \leq y < x \leq L} \prod_{i=1}^{n-1}  q^{ -   2 \xi_{i+1}^y \xi_{[1,i]}^x }.
$$
\end{proof}

\subsection{$j=1/2$ }
For $n$--species ASEP, all $\xi_i^x$ are $0$ or $1$, so this results in
$$
\mathbb{P}(\xi) \propto  \prod_{1 \leq y < x \leq L} \prod_{i=1}^{n-1}  q^{ -   2\xi_{[1,i]}^x\xi_{i+1}^y } .
$$
For $n=2$ it says that for the standard ASEP, this results in 
$$
\mathbb{P}(\xi) \propto \prod_{1 \leq x \leq L} q^{-2(x-1)\cdot \xi^x}
$$
For $n=3$ it reduces to
\begin{align*}
\prod_{1 \leq y < x \leq L}   q^{ -   2 \xi_{2}^y \xi_1^x - 2\xi_3^y(\xi^x_1+\xi^x_2) }  &= C' \prod_{1 \leq y < x \leq L} q^{-2\xi_1^x(1-\xi_1^y) } q^{-2\xi_3^y \xi^x_2}  \\
&=  C'\prod_{x=1}^{L}  q^{ -2x\xi_1^x } \prod_{1 \leq y < x \leq L}  q^{-2\xi_3^y \xi^x_2}.
\end{align*}

Compare this to Theorem 3.1 of \cite{BS}, which says that the reversible measures of two--species ASEP with asymmetry $q^{-1}$ should be 
\begin{align*}
\prod_{x=1}^{L}  q^{ 2x(\xi_1^x-\xi_3^x) } \prod_{1 \leq y < x \leq L} q^{ \xi^y_1 \xi^x_3 - \xi^y_3 \xi^x_1  } & = C \prod_{x=1}^{L}  q^{ 2x(\xi_1^x-\xi_3^x) } \prod_{1 \leq y < x \leq L} q^{- 2\xi^y_3 \xi^x_1  } \\
&= C'\prod_{x=1}^{L}  q^{ 2x\xi_1^x } \prod_{1 \leq y < x \leq L} q^{\xi^y_3(2- 2 \xi^x_1)  } ,
\end{align*}
which matches up to constants and switching $q\mapsto q^{-1}$.

\subsection{$n=2$}
When $n=2$ the formula is
\begin{align*}
C' \prod_{x=1}^L \frac{1}{  \q{\xi_1^x}^! \q{ \xi_2^x}^!   } q^{ ( \xi_1^x)^2/2  + (\xi_2^x)^2/2 } \prod_{1 \leq y < x \leq L}  q^{ -   2 \xi_{2}^y \xi_{1}^x  }  &= C''\prod_{x=1}^L \frac{1}{  \q{\xi_1^x}^! \q{ \xi_2^x}^!   } q^{ ( \xi_1^x)^2  } \prod_{1 \leq y < x \leq L}  q^{ -   2 \xi_{2}^y \xi_{1}^x  } \\
&=C'''\prod_{x=1}^L \frac{\q{2j}^!}{  \q{\xi_1^x}^! \q{ 2j-\xi_1^x}^!   } q^{ ( \xi_1^x)^2  } \prod_{1 \leq y < x \leq L}  q^{ -   4j \xi_{1}^x + 2\xi_1^y \xi_1^x  }\\
&= C''''\prod_{x=1}^L q^{-4jx \xi_1^x}\binomq{2j}{\xi_1^x},
\end{align*}
which matches the family of reversible measures for ASEP$(q,j)$, which are given by product measures with marginals
$$
\mathbb{P}^{(\alpha)}(\xi^x_1=n) \propto \alpha^n \binomq{2j}{n}q^{2n(1+j-2jx)}.
$$

\subsection{Stationary Measures and Duality}
In the equality defining duality,
$$
\mathbb{E}_x\left[ D(X(t),y)\right] = \mathbb{E}_y \left[ D(x,Y(t))\right],
$$
take $t \rightarrow \infty$ in the right--hand--side. If there is a unique stationary measure $Y(\infty)$, then the right--hand--side converges to $\mathbb{E}\left[ D(x,Y(\infty))\right],$ whose value does not depend on the initial state $y$. The same limit on the left--hand--side will yield $\mathbb{E}\left[ D(X(\infty),y)\right]$ whose value does not depend on the initial state $x$. Since these two sides are equal, 
\begin{align*}
&\mathbb{E}\left[ D(x,Y(\infty))\right] \text{ does not depend on } x\in X, \\
& \mathbb{E}\left[D(X(\infty),y)\right] \text{ does not depend on } y\in Y.
\end{align*}

Here are a few examples demonstrating this.

\underline{$n=2,j=1,L=2$}

The stationary measure for $\mathrm{ASEP}(q,1)$ with one particle are given by
$$
\mathbb{P}\left( 
\begin{array}{cc}
 2 & 2 \\
 1 & 2
\end{array}
\right) = \frac{q^{2}}{q^2+q^{-2}},
\quad  \quad
\mathbb{P}\left( 
\begin{array}{cc}
 2& 2\\
 2& 1
\end{array}
\right) = \frac{q^{-2}}{q^2+q^{-2}},
$$
and the duality from \eqref{CGRSDuality} is given by
$$
D'\left( 
\left( 
\begin{array}{cc}
 2& 2\\
 1 & 1
\end{array}
\right),
\left( 
\begin{array}{cc}
 2& 2\\
 2& 1
\end{array}
\right)
\right) = \frac{q^9}{q+q^{-1}},
\quad
D'\left( 
\left( 
\begin{array}{cc}
 2& 2\\
 1 & 1
\end{array}
\right),
\left( 
\begin{array}{cc}
 2& 2\\
 1 & 2
\end{array}
\right)
\right) = \frac{q^7}{q+q^{-1}}, 
$$
$$
D'\left( 
\left( 
\begin{array}{cc}
 1 & 2\\
 1 & 2
\end{array}
\right),
\left( 
\begin{array}{cc}
 2& 2\\
 1 & 2
\end{array}
\right)
\right) = q^6.
$$
Taking 
$$
x= \left( 
\begin{array}{cc}
 1 & 2\\
 1 & 2
\end{array}
\right), \quad
x' = \left( 
\begin{array}{cc}
 2& 2\\
 1 & 1
\end{array}
\right),
$$
predicts
$$
\mathbb{E}\left[ D'(x,Y(\infty)) \right] = \frac{q^8}{q^2+q^{-2}} = \frac{q^9+q^7}{(q+q^{-1})(q^2+q^{-2})} = \mathbb{E}\left[ D'(x',Y(\infty))\right]
$$
which is correct. Note that this example reveals a very small typo in the definitions of $D'$ in (35) and (37) of \cite{CGRS}. The expressions there have $-\eta^x$ in the exponent instead of $\eta^x$. That definition of $D'$ would instead imply
$$
\frac{q^4}{q^2+q^{-2}} = \frac{q^7+q^5}{(q+q^{-1})(q^2+q^{-2})}.
$$
The typo occurs due to a sign error in equation (174) in the proof of Lemma 7.3 of \cite{CGRS}, which has been confirmed by the authors in a private communication.

\underline{$n=2,j=1/2$}

For standard ASEP with one particle, there is an obvious stationary measure given by having a particle at every site with probability $1$. If $X(\infty)$ denotes this measure and $y_i,1 \leq i \leq L$ denotes the particle configuration with a particle at site $i$ and holes elsewhere, then the duality from \cite{S} is $D(X(\infty), y_i)=q^{2L},$ which does not depend on $i$. Note that this also reveals a very small typo in (3.12) from \cite{S}, which occurs due to a sign error in (2.14). This was corrected in Sections 3 and 2.3 of \cite{S2}, and also confirmed by the author in a private communication.

\underline{$n=3,j=1,L=2$}
Consider states with precisely one first--class and one second--class particle. This results in four possible states, and if they are ordered by 
$$
\begin{array}{cc}
2 & 3 \\
\underline{1} & \underline{3}
\end{array} \quad  \quad
\begin{array}{cc}
3 & 3 \\
\underline{1} & \underline{2} 
\end{array} \quad \quad
\begin{array}{cc}
3 & 3 \\
\underline{2} & \underline{1}
\end{array} \quad \quad
\begin{array}{cc}
3 & 2 \\
\underline{3} & \underline{1} 
\end{array} \quad  \quad
$$
then the generator is a $4\times 4$ matrix
$$
\mathcal{L}= \left(
\begin{array}{cccc}
 -q^3(q+q^{-1})^2 & q^3(1+q^2) & q(1+q^2) & 0 \\
q^7 & -(q+q^3+q^7) & q^3 & q \\
q^7  & q^5 & -(q+q^5+q^7) & q \\ 
0 & q^5(1+q^2) & q^{3}(1+q^2) &  -q^5(q+q^{-1})^2
\end{array}
\right)
$$
and the stationary measure is proportional to
$$
\left(
\begin{array}{cccc}
q^{-4} & (q+q^{-1})q^{-7} & (q+q^{-1})q^{-9} & q^{-12}.
\end{array}
\right)
$$

For $\eta$ equal to 
$$
\begin{array}{cc}
1 & 3 \\
\underline{1} & \underline{2}
\end{array}
\quad 
\text{ or }
\quad 
\begin{array}{cc}
2 & 3 \\
\underline{1} & \underline{1} 
\end{array},
$$
the duality is proportional to
$$
\left(q^6\q{2}^{-2} \quad q^7 \q{2}^{-3} \quad 0 \quad 0 \right) \quad \text{ or } \quad  \left(q^7 \q{2}^{-3} \quad q^8\q{2}^{-4} \quad q^9\q{2}^{-3} \quad 0 \right),
$$
which predicts that
$$
\q{2}^{-2}\left( q^2 + 1\right) = \q{2}^{-3} \left( q^3 + q + \q{2}\right),
$$
which is correct.




\bibliographystyle{alpha}

\end{document}